\documentclass{birkmult}
\usepackage{amsfonts}
\usepackage{amssymb}
\usepackage{amsmath}
\usepackage{graphicx}
\usepackage{epstopdf}

\newcommand\C{{\mathbb C}}
\newcommand\D{{\mathbb D}}

\newcommand\cS{{\mathcal S}}
\newcommand\cK{{\mathcal K}}

\newcommand\N{{\mathbb N}}

\newcommand\R{{\mathbb R}}

\newcommand\Z{{\mathbb Z}}

\newcommand\spec{{\rm spec}\,}  
\newcommand\specn{{\rm spec}}   

\newcommand\ri{{\rm i}}
\newcommand\re{{\rm e}}

\newcommand\Ome{\Omega}

\newcommand\intt{{\rm int}}
\newcommand\per{\rm per}

\newcommand\ovD{{\overline{\mathbb D}}}

 \newtheorem{thm}{Theorem}[section]
 \newtheorem{cor}[thm]{Corollary}
 \newtheorem{lem}[thm]{Lemma}
 \newtheorem{prop}[thm]{Proposition}
 \theoremstyle{definition}
 
 \theoremstyle{remark}
 \newtheorem{rem}[thm]{Remark}
 
 \numberwithin{equation}{section}

\begin{document}
%
%
%
%
%
%
%
%
%
\title[Feinberg-Zee Random Hopping Matrix]
 {On Symmetries of the Feinberg-Zee Random Hopping Matrix}
\author[Simon Chandler-Wilde]{Simon N.~Chandler-Wilde}

\address{%
Department of Mathematics and Statistics\\
University of Reading\\
Reading, RG6 6AX\\
UK}
\email{S.N.Chandler-Wilde@reading.ac.uk}

\author{Raffael Hagger}
\address{
Institute of Mathematics\\
Hamburg University of Technology\\
Schwarzenbergstr. 95 E\\
21073 Hamburg\\
GERMANY\\
\textit{now at:}\\
Institute for Analysis\\
Leibniz Universit\"at Hannover\\
Welfengarten 1\\
30167 Hannover\\
GERMANY}
\email{raffael.hagger@math.uni-hannover.de}
\subjclass{Primary 47B80; Secondary 37F10, 47A10, 47B36, 65F15}

\keywords{random operator, Jacobi operator, non-selfadjoint operator, spectral theory, fractal, Julia set}

\date{August 27, 2015}
\dedicatory{Dedicated to Roland Duduchava on the occasion of his 70th birthday}

\begin{abstract}
In this paper we study the spectrum $\Sigma$ of the infinite Feinberg-Zee random hopping matrix, a tridiagonal matrix with zeros on the main diagonal and random $\pm 1$'s on the first sub- and super-diagonals; the study of this non-selfadjoint random matrix was initiated in Feinberg and Zee ({\it Phys. Rev. E} {\bf 59} (1999), 6433--6443). Recently Hagger ({\em Random Matrices: Theory Appl.}, {\bf 4} 1550016 (2015)) has shown that the so-called {\em periodic part} $\Sigma_\pi$ of $\Sigma$, conjectured to be the whole of $\Sigma$ and known to include the unit disk, satisfies $p^{-1}(\Sigma_\pi) \subset \Sigma_\pi$ for an infinite class $\cS$ of monic polynomials $p$. In this paper we make very explicit the membership of $\cS$, in particular showing that it includes $P_m(\lambda) = \lambda U_{m-1}(\lambda/2)$, for $m\geq 2$, where $U_n(x)$ is the Chebychev polynomial of the second kind of degree $n$. We also explore implications of these inverse polynomial mappings, for example showing that $\Sigma_\pi$ is the closure of its interior, and contains the filled Julia sets of infinitely many $p\in \cS$, including those of $P_m$, this partially answering a conjecture of the second author.
\end{abstract}

\maketitle
\section{Introduction} \label{sec:intro}

In this paper we study spectral properties of infinite matrices of the form
\begin{equation} \label{eq:matrix}
A_c = \begin{pmatrix} \ddots & \ddots & & & \\ \ddots & 0 & 1 & & \\ & c_0 & \fbox{$0$} & 1 & \\ & & c_1 & 0 & \ddots \\ & & & \ddots & \ddots \end{pmatrix},
\end{equation}
where $c\in \Omega := \{\pm 1\}^\Z$ is an infinite sequence of $\pm 1$'s, and the box marks the entry at $(0,0)$. Let $\ell^2$ denote the linear space of those complex-valued sequences $\phi:\Z\to \C$ for which $\|\phi\|_2 := \{\sum_{n\in\Z}|\phi_n|^2\}^{1/2}<\infty$, a Hilbert space equipped with the norm $\|\cdot\|_2$. Then to each matrix $A_c$ with $c\in \Ome$ corresponds a bounded linear mapping $\ell^2\mapsto\ell^2$, which we denote again by $A_c$, given  by the rule
$$
(A_c \phi)_m = c_m \phi_{m-1} + \phi_{m+1}, \quad m\in\Z,
$$
for $\phi\in\ell^2$.

Following \cite{CWDavies2011} we will term \eqref{eq:matrix} a {\em Feinberg-Zee hopping matrix}. Further,  in the case where each $c_m$ is an independent realisation of a random variable with probability measure whose support is $\{-1,1\}$, we will term $A_c$ a Feinberg-Zee {\em random} hopping matrix, this particular non-selfadjoint random matrix studied previously in \cite{FeinZee97,FeinZee99,CicutaContediniMolinari2000,HolzOrlZee,CCL,CWDavies2011,CCL2,Hagger:NumRange,Hagger:dense,Hagger:symmetries}. \footnote{These random hopping matrices appear to have been studied initially in \cite{FeinZee97}, in which paper the first superdiagonal is also a sequence of random $\pm1$'s. But it is no loss of generality to restrict attention to matrices of the form \eqref{eq:matrix} as the case where the superdiagonal is also random can be reduced to \eqref{eq:matrix} by a simple gauge transformation; see \cite{FeinZee97} or \cite[Lemma 3.2, Theorem 5.1]{CCL2}.} The spectrum of a realisation $A_c$ of this random hopping matrix is given, almost surely, by (e.g., \cite{CCL})
\begin{equation} \label{eq:spec}
\spec A_c = \Sigma := \bigcup_{b\in \Ome} \spec A_b.
\end{equation}
Here $\spec A_b$ denotes the spectrum of $A_b$ as an operator on $\ell^2$. Note that \eqref{eq:spec} implies that $\Sigma$ is closed.

 Equation \eqref{eq:spec} holds whenever $c\in \Ome$ is {\em pseudo-ergodic}, which means simply that every finite sequence of $\pm 1$'s appears as a consecutive sequence somewhere in the infinite vector $c$;  it is easy to see that $c$ is pseudo-ergodic almost surely if $c$ is random.  The concept of pseudo-ergodicity dates back to \cite{Davies2001:PseudoErg}, as do the arguments that \eqref{eq:spec} holds, or see \cite{CCL2} for \eqref{eq:spec} derived as a special case of more general limit operator results.

Many of the above cited papers are concerned primarily with computing upper and lower bounds on $\Sigma$. A standard upper bound for the spectrum is provided by the numerical range. It is shown in \cite{CCL2} that, if $c\in \Ome$ is pseudo-ergodic, its numerical range $W(A_c)$, defined by $W(A_c) :=\{(A_c \phi,\phi):\phi\in \ell^2,\ \|\phi\|_2=1\}$, where $(\cdot,\cdot)$ is the inner product on $\ell^2$, is given by
\begin{eqnarray} \label{eq:nr}
W(A_c) = \Delta:= \{x+\ri y: x,y\in \R, \; |x|+|y| < 2\}.
\end{eqnarray}
 This gives the upper bound that $\Sigma \subset \overline{\Delta}$, the closure of $\Delta$. Other, sharper upper bounds on $\Sigma$ are discussed in Section \ref{sec:pre} below.

This current paper is related to the problem of computing lower bounds for $\Sigma$ via \eqref{eq:spec}. If $b\in \Ome$ is constant then $A_b$ is a Laurent matrix and $\spec A_b = [-2,2]$ if $b_m\equiv 1$, while $\spec A_b=\ri[-2,2]$ if $b_m\equiv -1$; thus, by \eqref{eq:spec}, $\pi_1:= [-2,2]\cup \ri [-2,2]\subset \Sigma$. Generalising this, if $b\in \Ome$ is periodic with period $n$ then $\spec A_b$ is the union of a finite number of analytic arcs which can be computed by calculating eigenvalues of $n\times n$ matrices (see Lemma \ref{lem:spper} below). And, by \eqref{eq:spec}, $\pi_n\subset \Sigma$, where $\pi_n$ is the union of $\spec A_b$ over all $b$ with period $n$. This implies, since $\Sigma$ is closed, that
\begin{equation} \label{eq:piinf}
\Sigma_\pi := \overline{\pi_\infty} \subset \Sigma,
\end{equation}
where $\pi_\infty := \cup_{n\in\N} \pi_n$.

We will call $\Sigma_\pi$ the {\em periodic part} of $\Sigma$, noting that \cite{CCL} conjectures that equality holds in \eqref{eq:piinf}, i.e.~that $\pi_\infty$ is dense in $\Sigma$ and $\Sigma_\pi=\Sigma$. Whether or not this holds is an open problem, but it has been shown in \cite{CWDavies2011} that $\pi_\infty$ is dense in the open unit disk $\D:=\{\lambda\in \C:|\lambda|<1\}$, so that
\begin{equation} \label{eq:unit}
\ovD \subset \Sigma_\pi\subset \Sigma.
\end{equation}

For a polynomial $p$ and $S\subset \C$, we define, as usual, $p(S):= \{p(\lambda):\lambda\in S\}$ and $p^{-1}(S):= \{\lambda\in \C: p(\lambda)\in S\}$. (We will use throughout that if $S$ is open then $p^{-1}(S)$ is open ($p$ is continuous) and, if $p$ is non-constant, then $p(S)$ is also open, e.g., \cite[Theorem 10.32]{rudinrealcomplex}.) The proof of \eqref{eq:unit} in \cite{CWDavies2011} depends on the result, in the case $p(\lambda) = \lambda^2$, that
\begin{equation} \label{eq:pm1}
p^{-1}(\pi_\infty) \subset \pi_\infty, \; \mbox{ so that also } p^{-1}(\Sigma_\pi)\subset \Sigma_\pi.
\end{equation}
This implies that $S_n\subset \pi_\infty$, for $n=0,1,...$, where $S_0:=[-2,2]$ and $S_{n}:=p^{-1}(S_{n-1})$, for $n\in\N$. Thus $\cup_{n\in\N}S_n$, which is dense in $\ovD$, is also in $\pi_\infty$, giving \eqref{eq:unit}.

Hagger \cite{Hagger:symmetries} makes a large generalisation of the results of \cite{CWDavies2011}, showing the existence of an infinite family, $\cS$, containing monic polynomials of arbitrarily high degree, for which \eqref{eq:pm1} holds. For each of these polynomials $p$ let
\begin{equation} \label{eq:Up}
U(p) := \bigcup_{n=1}^\infty p^{-n}(\D).
\end{equation}
(Here $p^{-2}(S) := p^{-1}(p^{-1}(S))$, $p^{-3}(S) := p^{-1}(p^{-2}(S))$, etc.) Hagger \cite{Hagger:symmetries} observes that, as a consequence of \eqref{eq:unit} and \eqref{eq:pm1}, $U(p)\subset \Sigma_\pi$. He also notes that standard results of complex dynamics (e.g., \cite[Corollary 14.8]{Falconer03}) imply that $J(p)\subset \overline{U(p)}$, so that $J(p)\subset \Sigma_\pi$; here $J(p)$ denotes the Julia set of the polynomial $p$. (Where $p^2(\lambda):=p(p(\lambda))$, $p^3(\lambda):=p(p^2(\lambda))$, etc., we recall \cite{Falconer03} that the {\em filled Julia set} $K(p)$ of a polynomial $p$ of degree $\geq 2$ is the compact set of those $\lambda\in \C$ for which the sequence $(p^n(\lambda))_{n\in\N}$, the {\em orbit} of $\lambda$, is bounded. Further, the boundary of $K(p)$, $J(p) := \partial K(p)\subset K(p)$,  is the {\em Julia set} of $p$.)

The definition of the set $\cS$ in \cite{Hagger:symmetries}, while constructive, is rather indirect. The first contribution of this paper (Section \ref{sec:poly}) is to make explicit the membership of $\cS$. As a consequence we show, in particular, that $P_m\in \cS$, for $m=2,3,...$, where $P_m(\lambda) := \lambda U_{m-1}(\lambda/2)$, and $U_n$ is the Chebychev polynomial of the second kind of degree $n$ \cite{AS}.

The second contribution of this paper (Section \ref{sec:interior}) is to say more about the interior points of $\Sigma_\pi$. Previous calculations of large subsets of $\pi_\infty$, precisely calculations of $\pi_n$
for $n$ as large as 30 \cite{CCL,CCL2}, suggest that $\Sigma_\pi$ fills most of the square $\Delta$, but $\intt(\Sigma_\pi)$, the interior of $\Sigma_\pi$, is known only to contain $\D$. Using that the whole family $\{P_m:m\geq 2\}\subset \cS$, we prove that $(-2,2)\subset  \intt(\Sigma_\pi)$. This result is then used to show that 
$\Sigma_\pi$ is the closure of its interior. Using that $p^{-1}(\D)\subset \Sigma_\pi$, for $p\in \cS$, we also, in Section \ref{subsec:b}, construct new explicit subsets of $\Sigma_\pi$ and its interior; in particular, extending \eqref{eq:unit}, we show that $\alpha \ovD \subset \Sigma_\pi$ for $\alpha = 1.1$.

In the final Section \ref{sec:julia} of the paper we address a conjecture of Hagger \cite{Hagger:symmetries} that, not only for every $p\in \cS$ is $J(p)\subset \overline{U(p)}$ (which implies $J(p)\subset \Sigma_\pi$), but also the filled Julia set $K(p)\subset \overline{U(p)}$. This is a stronger result as, while the compact set $J(p)$ has empty interior \cite[Summary 14.12]{Falconer03}, $K(p)$ contains, in addition to $J(p)$,  all the bounded components of the open Fatou set $F(p):= \C\setminus J(p)$. We show, by a counter-example, that this conjecture is false. But, positively, we conjecture that $K(p)\subset \Sigma_\pi$ for all $p\in \cS$, and we prove that this is true for a large subset of $\cS$, in particular that $K(P_m)\subset \Sigma_\pi$ for $m\geq 2$.

The results in this paper provide new information on the almost sure spectrum $\Sigma\supset \Sigma_\pi$ of the bi-infinite Feinberg-Zee random hopping matrix. They are also relevant to the study of the spectra of the corresponding finite matrices. For $n\in \N$ let $V_n$ denote the set of $n\times n$ matrices of the form \eqref{eq:matrix}, so that $V_1 := \{(0)\}$ and, for $n\geq 2$, $V_n:= \{A_k^{(n)}:k=(k_1,...,k_n)\in \{\pm 1\}^n\}$, where
\begin{equation} \label{eq:matrix2}
A^{(n)}_k:= \left(\begin{array}{ccccc}
0&1&&\\
k_1&0&\ddots&\\
&\ddots&\ddots& 1\\
&& k_{n-1} & 0
\end{array}\right).
\end{equation}
(This notation will be convenient, but note that $A_k^{(n)}$ is independent of the last component of $k$.) Then $\spec A_k^{(n)}$ is the set of eigenvalues of the matrix $A_k^{(n)}$. Let
\begin{equation} \label{eq:sigma}
\sigma_n := \bigcup_{A\in V_n} \spec A, \mbox{ for }n\in\N, \; \mbox{ and } \sigma_\infty := \bigcup_{n=1}^\infty \sigma_n,
\end{equation}
so that $\sigma_\infty$ is the union of all eigenvalues of finite matrices of the form \eqref{eq:matrix2}. Then, connecting spectra of finite and infinite matrices, it has been shown in \cite{CCL2} that $\sigma_n\subset \pi_{2n+2}$, for $n\in\N$, so that $\sigma_\infty\subset \pi_\infty \subset \Sigma_\pi$. Further, \cite{Hagger:dense} shows that $\sigma_\infty$ is dense in $\pi_\infty$, so that $\Sigma_\pi = \overline{\sigma_\infty}$. In Section \ref{subsec:a} we build on and extend these results, making a surprising connection between the eigenvalues of the finite matrices \eqref{eq:matrix2} and the spectra of the periodic operators associated to the polynomials in $\cS$. The result we prove (Theorem \ref{thm:dense}), is key to the later arguments in Section \ref{sec:interior}.

\section{Preliminaries and previous work} \label{sec:pre}

\paragraph{Notions of set convergence.} We will say something below about set sequences, sequences approximating $\Sigma$ and $\Sigma_\pi$ from above and below, respectively. We will measure convergence in the standard Haussdorf metric $d(\cdot,\cdot)$ \cite[Section 28]{Hausdorff} (or see \cite{HaRoSi2}) on the space $\C^C$ of compact subsets of $\C$. We will write, for a sequence $(S_n)\subset \C^C$ and $S\in \C^C$, that $S_n\searrow S$ if $d(S_n,S)\to 0$ as $n\to \infty$ and $S\subset S_n$ for each $n$; equivalently, $S_n\searrow S$ as $n\to\infty$ if $S\subset S_n$ for each $n$ and, for every $\epsilon >0$, $S_n\subset S+\epsilon \D$, for all sufficiently large $n$. Similarly, we will write that $S_n\nearrow S$ if $d(S_n,S)\to 0$ as $n\to \infty$ and $S_n\subset S$ for each $n$; equivalently, $S_n\nearrow S$ if $S_n\subset S$ for each $n$ and, for every $\epsilon >0$, $S\subset S_n+\epsilon \D$, for all sufficiently large $n$. The following observation, which follows immediately from \cite[Proposition 3.6]{HaRoSi2} (or see \cite[Lemma 4.15]{CWLi2015:Coburn}),
will be useful:
\begin{lem} \label{eq:setconv}
If $S_1\subset S_2 \subset ... \subset \C$ are closed and
$
S_\infty:= \bigcup_{n=1}^\infty S_n
$
is bounded, then $S_n\nearrow \overline{S_\infty}$, as $n\to\infty$.
\end{lem}

\paragraph{Spectra of periodic operators.} We will need explicit formulae for the spectra of operators $A_c$ with $c\in \Ome$ in the case when $c$ is periodic. For $k=(k_1,...,k_n)\in \{\pm 1\}^n$, let $A_k^{\per}$ denote $A_c$ in the case that $c_{m+n} = c_n$, for $m\in \Z$, and $c_m=k_m$, for $m=1,2,...,n$. For $n\in \N$ let $I_n$ denote the order $n$ identity matrix, $R_n$ the $n\times n$ matrix which is zero except for the entry 1 in row $n$, column $1$, and let $R_n^T$ denote the transpose of $R_n$. For $n\in \N$, $k\in \{\pm 1\}^n$, and $\varphi\in \R$, let $a_k(\varphi) := A_k^{(n)} + \re^{-\ri \varphi} R_n + k_n \re^{\ri \varphi} R_n^T$. The following characterisation of the spectra of periodic operators is well-known (see Lemma 1 and the discussion in \cite{Hagger:dense}).

\begin{lem} \label{lem:spper}
For $n\in \N$ and $k\in \{\pm 1\}^n$,
$$
\spec A_k^{\per} = \{\lambda\in \C: \det(a_k(\varphi)-\lambda I_n) =0 \mbox{ for some }\varphi\in [0,2\pi)\}.
$$
\end{lem}

Key to our arguments will be an explicit expansion for the determinant in the above lemma, expressed in terms of the following notation.
For $n\in \N$, $k=(k_1,...,k_n)\in \{\pm 1\}^n$, and $\lambda\in \C$, let
\begin{equation}\label{eq:pdef3}
q_k(\lambda) :=          \left|\begin{array}{cccc}
                                      \lambda & 1 \\
                                       k_1 & \ddots & \ddots\\
                                       & \ddots & \ddots & 1\\
                                       & & k_{n} & \lambda
                                     \end{array}\right|.
\end{equation}
For $i,j\in \Z$ and $\lambda\in \C$, let $k(i:j):= (k_i,...,k_j)$, for $1\leq i\leq j\leq n$, and define
\begin{equation} \label{eq:ij}
q_{k(i:j)}(\lambda):= \left\{\begin{array}{cc}
                      \lambda, & \mbox{ if } i-j = 1, \\
                      1, & \mbox{ if } i-j = 2, \\
                      0, & \mbox{ if } i-j = 3.
                    \end{array}\right.
\end{equation}
 Then,  for $n\in \N$ and $k\in \{\pm 1\}^n$, expanding the determinant \eqref{eq:pdef3} by Laplace's rule by the first row and by the last row, we see that
\begin{equation} \label{eq:qk}
q_k(\lambda) =\lambda q_{k(2:n)}(\lambda)-k_1 q_{k(3:n)}(\lambda) = \lambda q_{k(1:n-1)}(\lambda)-k_{n} q_{k(1:n-2)}(\lambda).
\end{equation}
The following lemma follows easily by induction on $n$, using \eqref{eq:qk}. The bounds on $q_k$ stated are used later in Corollary \ref{cor:pkbound}.
\begin{lem} \label{lem:qkbound}
If $k=(k_1,...k_n)\in \{\pm 1\}^n$, for some $n\in \N$, then $q_k$ is a monic polynomial of degree $n+1$, and $q_k$ is even and $q_k(0)=\pm 1$ if $n$ is odd, $q_k$ is odd and $q_k(0)=0$ if $n$ is even. Further, $|q_k(\lambda)| \geq |q_{k(1:n-1)}(\lambda)|+1$, $|q_k(\lambda)| \geq |q_{k(2:n)}(\lambda)|+1$, and $|q_k(\lambda)|\geq n+2$, for
$|\lambda|\geq 2$.
\end{lem}

For $n\in \N$ let $J_n$ denote the $n\times n$ flip matrix, that is the $n\times n$ matrix with entry $\delta_{i,n+1-j}$ in row $i$, column $j$, where $\delta_{i,j}$ is the Kronecker delta. Then $J_n^2 = I_n$ so that $(\det J_n)^2 = 1$.
For $k=(k_1,...,k_n)\in \{\pm 1\}^n$, let $k^\prime  := k J_n = (k_n,...,k_1)$. The first part of the following lemma is essentially  a particular instance of a general property of determinants.
\begin{lem} \label{lem:qksym}
If $k\in \{\pm 1\}^n$, for some $n\in \N$, and $\ell=k^\prime$, then $q_k= q_\ell$; if $\ell=-k$, then
\begin{equation} \label{eq:qel}
q_\ell(\lambda) = \ri^{-n-1} q_k(\ri \lambda).
\end{equation}
\end{lem}
\begin{proof} Suppose first that $\ell=k^\prime$. Then $q_k(\lambda)$ given by \eqref{eq:qk} is the determinant of a matrix $A$, and $q_\ell(\lambda)$ is the determinant of $J_{n+1} A^T J_{n+1}$, so that $q_\ell(\lambda) =(\det J_{n+1})^2\det A= q_k(\lambda)$. That \eqref{eq:qel} holds if $\ell=-k$ can be shown by an easy induction on $n$, using \eqref{eq:qk}, or directly by a gauge transformation.
\end{proof}

Here is the announced explicit expression for the determinant in Lemma \ref{lem:spper}.

\begin{lem} \label{lem:det}
\cite[Lemma 3]{Hagger:symmetries}
For $n\in \N$, $k\in \{\pm 1\}^n$, $\lambda\in \C$, and $\varphi\in [0,2\pi)$,
$$
(-1)^n\det(a_k(\varphi)-\lambda I_n) = p_k(\lambda) - \re^{\ri\varphi} \prod_{j=1}^n k_j - \re^{-\ri\varphi},
$$
where $p_k$ is a monic polynomial of degree $n$ given by
\begin{equation} \label{eq:pdef2}
(-1)^n p_k(\lambda) =   q_{k(1:n-1)}(-\lambda) -k_n q_{k(2:n-2)}(-\lambda).
\end{equation}
Further, $p_k$ is odd (even) if $n$ is odd (even).
\end{lem}
Since, from the above lemmas, $p_k$ is odd (even) and $q_k$ even (odd) if $n$ is odd (even), \eqref{eq:pdef2} implies that
\begin{eqnarray} \label{eq:pk2}
p_k(\lambda) &=&   q_{k(1:n-1)}(\lambda) -k_n q_{k(2:n-2)}(\lambda) \\ \label{eq:pk2a}
& = & q_{k(1:n-1)}(\lambda) + q_{k(2:n)}(\lambda) - \lambda q_{k(2:n-1)}(\lambda),
\end{eqnarray}
this last equation obtained using \eqref{eq:qk}. The following lemma, proved using these representations, makes clear that many different vectors $k$ correspond to the same polynomial $p_k$.

\begin{lem}\label{lem:pksym}
If $k=(k_1,...k_n)\in \{\pm 1\}^n$, for some $n\in \N$, and $\ell=k^\prime$ or $\ell$ is a cyclic permutation of $k$, then $p_k= p_\ell$. If $\ell=-k$ then $p_\ell(\lambda) = \ri^{-n} p_k(\ri \lambda)$.
\end{lem}
\begin{proof}
Using \eqref{eq:pk2} and \eqref{eq:qk} we see that
\begin{eqnarray*}
p_\ell(\lambda)-p_k(\lambda) &=& q_{\ell(1:n-1)}(\lambda) - q_{k(1:n-1)}(\lambda) + k_n q_{k(2:n-2)}(\lambda) - \ell_n q_{\ell(2:n-2)}(\lambda)\\
& = & \lambda \left(q_{\ell(2:n-1)}(\lambda) - q_{k(1:n-2)}(\lambda)\right) -\ell_1 q_{\ell(3:n-1)}(\lambda)\\
& & \hspace{3ex} +k_{n-1} q_{k(1:n-3)}(\lambda)+ k_n q_{k(2:n-2)}(\lambda) - \ell_n q_{\ell(2:n-2)}(\lambda).
\end{eqnarray*}
If $\ell$ is a cyclic shift of $k$, i.e., $\ell_j=k_{j-1}$, $j=2,...,n$, and $\ell_1=k_n$, then the right-hand side is identically zero. Thus $p_\ell=p_k$ if $\ell$ is a cyclic permutation of $k$.

If $\ell=k^\prime$ then that $p_k=p_\ell$ follows from \eqref{eq:pk2a} and Lemma \ref{lem:qksym}. If $\ell=-k$ then that  $p_\ell(\lambda) = \ri^{-n} p_k(\ri \lambda)$ follows from  \eqref{eq:pk2} and Lemma \ref{lem:qksym}.
\end{proof}

Call $k\in \{\pm 1\}^n$ {\em even} if $\prod_{j=1}^n k_j = 1$, and {\em odd} if $\prod_{j=1}^n k_j = -1$. Then
\cite[Corollary 5]{Hagger:symmetries}, it is immediate from Lemmas \ref{lem:spper} and \ref{lem:det} that
\begin{equation} \label{eq:per}
\spec A_k^{\per} = p_k^{-1}([-2,2]), \mbox{ if $k$ is even}, \; \spec A_k^{\per} = p_k^{-1}(\ri[-2,2]), \mbox{ if $k$ is odd.}
\end{equation}

\paragraph{Complex dynamics.}  In  Section \ref{sec:julia} below we show that filled Julia sets, $K(p)$, of particular polynomials $p$, are contained in the periodic part $\Sigma_\pi$ of the almost sure spectrum of the Feinberg-Zee random hopping matrix. To articulate and prove these results
we will need terminology and results from complex dynamics.

Throughout this section $p$ denotes a polynomial of degree $\geq 2$. We have defined above the compact set that is the filled Julia set $K(p)$, the Julia set $J(p) = \partial K(p) \subset K(p)$, the Fatou set $F(p)$ (the open set that is the complement of $J(p)$), and the orbit of $z\in \C$. It is easy to see that, if $z\not\in K(p)$, i.e., the orbit of $z$ is not bounded, then $p^n(z)\to\infty$ as $n\to\infty$, i.e., $z\in A_p(\infty)$, the {\em basin of attraction of infinity}. We call $S\subset \C$ {\em invariant} if $p(S)=S$, and {\em completely invariant} if both $S$ and its complement are invariant, which holds iff $p^{-1}(S)=S$. Clearly $A_p(\infty)$ is completely invariant.

 We call $z$ a {\em fixed point} of $p$ if $p(z)=z$ and a {\em periodic point} if $p^n(z)=z$, for some $n\in\N$, in which case the finite sequence $(z_0,z_1,...,z_{n-1})$, where $z_0=z$, $z_1=p(z_0)$, ..., is the {\em cycle} of the periodic point $z$. We say that $z$ is an {\em attracting fixed point} if $|p^\prime(z)|<1$, a {\em repelling fixed point} if $|p^\prime(z)|>1$, and a {\em neutral fixed point} if $|p^\prime(z)|=1$. Generalising, we say that a periodic point $z$ is attracting/repelling/neutral if $|P^\prime(z)|<1$/$>1$/$=1$, where $P=p^n$. By the chain rule, $P^\prime(z) = p^\prime(z_0)...p^\prime(z_{n-1})$, where $z_0=z$ and $z_j:= p(z_{j-1})$, $j=1,...,n-1$.

The value $\gamma = P^\prime(z)$ is the {\em multiplier} of the neutral periodic point $z$ (clearly $|\gamma|=1$).
If $z$ is a neutral periodic point with multiplier $\gamma$ we say that it is {\em rationally neutral} if $\gamma^N=1$ for some $N\in \N$, otherwise we say that it is {\em irrationally neutral}. We call $z$ a {\em critical point} of $p$ if $p^\prime(z)=0$.

If $w$ is an attracting periodic point we denote by $A_p(w)$ the {\em basin of attraction} of the cycle $C= \{z_0,...,z_{n-1}\}$ of $z$, by which we mean
$A_p(w) := \{z\in \C: d(p^n(z),C)\to 0 \mbox{ as } n\to\infty\}$.
Here, for $S\subset\C$ and $z\in\C$, $d(z,S):= \inf_{w\in S}|z-w|$. It is easy to see that $A_p(w)$ contains some neighbourhood of $C$, and hence that $A_p(w)$ is open.

We will make use of standard properties of the Julia set captured in the following theorem. We recall that a family $F$ of analytic functions is {\em normal} at a point $z\in \C$ if, in some fixed neighbourhood $N$ of $z$, each $f\in F$ is analytic and every sequence drawn from $F$ has a subsequence that is convergent uniformly either to some analytic function or to $\infty$.

\begin{thm} \cite[Summary 14.12]{Falconer03} $J(p)$ is compact with no isolated points, is uncountable, and has empty interior. $J(p)$ is completely invariant, $J(p)=J(p^n)$, for every $n\in\N$,
\begin{equation} \label{eq:nn}
J(p) = \{z\in\C: \mbox{ the family } \{p^1,p^2,...\} \mbox{ is not normal at }z\},
\end{equation}
$J(p)$ is the closure of the repelling periodic points of $p$, and, for all except at most one $z\in \C$,
\begin{equation} \label{eq:Jpi}
J(p)\subset \overline{\bigcup_{n=1}^\infty p^{-n}(\{z\})}.
\end{equation}
\end{thm}

The Fatou set $F(p)$ has one unbounded component $U$. It follows from \eqref{eq:nn} that $U\subset A_p(\infty)$; indeed, $U=A_p(\infty)$ as a consequence of the maximum principle \cite{CarlesonGamelin92}. It may happen that this is the only component of $F(p)$ so that $A_p(\infty)=F(p)$. This is the case if $k=(1,1)$ and $p(z)=p_k(z)=z^2-2$, when $K(p)=J(p)=[-2,2]$ \cite[p.~55]{CarlesonGamelin92}. If $F(p)$ has more than one component it either has two components (for example, if $k=(-1,1)$ and $p(z)=p_k(z)=z^2$, when $K(p)=\ovD$ and $J(p)=\partial \D$), or infinitely many components \cite[Theorem IV.1.2]{CarlesonGamelin92}.
It follows from \eqref{eq:Jpi} that $J(p) = \partial A_p(\infty)$ and $J(p) = \partial A_p(w)$ if $w$ is an attracting fixed point or periodic point \cite[Theorem III.2.1]{CarlesonGamelin92}. Arguing similarly \cite[Theorem 1.7]{CarlesonGamelin92}, $J(p) = \partial F_B(p)$, where $F_B(p):= \intt(K(p))= F(p)\setminus A_p(\infty)$, so that $K(p)= \overline{F_B(p)}$.

Because $J(p)$ is completely invariant and $p$ is an open map, the image $V=p(U)$ of any  component $U$ of $F(p)$ is also a component of $F(p)$. Now consider the orbit of $U$, i.e. $(p(U))_{n=1}^\infty$. The following statement of possible behaviours is essentially Sullivan's theorem \cite[pp.~69-71]{CarlesonGamelin92}.
\begin{thm} \label{thm:sullivan}
Let $U$ be a component of $F(p)$. Then one of the following cases holds:

i) $p^n(U)=U$, for some $n\in \N$, in which case we call $U$ a {\em periodic component} of $F(p)$, and call the smallest $n$ for which $p^n(U)=U$ the {\em period} of $U$. (If $n=1$, when $U$ is invariant, we also term $U$ a {\em fixed component} of $F(p)$.)

ii) $p^r(U)$ is a periodic component of $F(p)$ for some $r\in \N$, in which case we say that $U$ is a {\em preperiodic component} of $F(p)$.
\end{thm}

The above theorem makes clear that the orbit of every component $U$ of $F(p)$ enters a periodic cycle after a finite number of steps. To understand the eventual fate under iterations of $p$ of the components of the Fatou set it is helpful to understand the possible behaviours of a periodic component. This is achieved in the {\em classification theorem} (e.g.~\cite{CarlesonGamelin92}). To state this theorem we introduce further terminology. Let us call a fixed component $U$ of $F(p)$ a {\em parabolic} component if there exists a neutral fixed point $w\in \partial U$ with multiplier 1 such that the orbit of every $z\in U$ converges to $w$. Call a fixed component $U$ of $F(p)$ a {\em Siegel disk} if $p$ is conjugate to an irrational rotation on $U$, which means that there exists a conformal mapping $\varphi:U\to V$ with $0\in V$  and an irrational $\theta\in \R$ such that
\begin{equation} \label{eq:st}
\varphi(p(z)) = g(\varphi(z)) = \gamma \varphi(z), \quad z\in U,
\end{equation}
where $g(w) = \gamma w$, $w\in V$, and $\gamma = \exp(2\pi \ri \theta)$. It is easy to see that, for $w\in U$, $p(w)=w$ iff $w=\varphi^{-1}(0)$, and $p^\prime(w)=\gamma$. Thus every Siegel disk contains a unique irrationally neutral fixed point (the Siegel disk fixed point).

\begin{thm} \label{thm:class}
{\em Classification Theorem \cite[Theorem IV.2.1]{CarlesonGamelin92}.}
If $U$ is a periodic component of $F(p)$ with period $n\in \N$, in which case $U$ is also a component of $F(p^n)=F(p)$, then exactly one of the following holds:\footnote{The result as stated for rational $p$ in \cite{CarlesonGamelin92} gives a 4th option, that $U$ is a {\em Herman ring component} of $F(p^n)$. This option is excluded if $p$ is a polynomial \cite[p.~166]{Milnor}.}

a) $U$ contains an attracting periodic point $w$ which is an attracting fixed point of $p^n$, and $U\subset A_p(w)$;

b) $U$ is a parabolic component of $F(p^n)$;

c) $U$ is a Siegel disk component of $F(p^n)$.
\end{thm}

The following proposition relates the above cases to critical points of $p$ (see \cite[Theorems III 2.2 and 2.3, pp.~83-84]{CarlesonGamelin92}):

\begin{prop} \label{prop:crit}
If $U$ is a periodic component of $F(p)$ with period $n$ then either: (i) $U$ is a parabolic component of $F(p^n)$ or contains an attracting periodic point, in which case $\cup_{m=1}^n p^m(U)$ contains a critical point of $p$;  or (ii) $U$ is a Siegel disk component of $F(p^n)$ and there is a critical point $w$ of $p$ such that the orbit of $w$ is dense in $\partial U$.
\end{prop}

The following proposition will do the work for us in Section \ref{sec:julia}.

\begin{prop} \label{prop:cd2a}
Suppose that $S\subset \C$ is bounded, open and simply-connected, that $T\subset S$ is closed, and that the orbit of every critical point in $K(p)$ eventually lies in $T$. Then
$$
K(p)\subset G:= \overline{\bigcup_{n=1}^\infty p^{-n}(S)}.
$$
\end{prop}
\begin{proof} That $z\in G$ if $z\in J(p)$ follows from \eqref{eq:Jpi}. If $z\in  F_B(p)$ then, by Theorems \ref{thm:sullivan} and \ref{thm:class}, after a finite number of iterations the orbit of $z$ is in a periodic component of $F(p)$ that is parabolic, part of the domain of attraction of an attracting periodic point, or is a Siegel disk. In the first two cases it follows that $d(p^n(z), C)\to 0$ as $n\to\infty$ for some cycle $C$, but also $d(p^n(w),C)\to 0$ for some critical point $w$ by Proposition \ref{prop:crit}. This last implies that $C\cap T$ is non-empty, and so $p^n(z)\in S$ for some $n$. In the case that the orbit of $z$ is eventually in a Siegel disk then also $p^n(z)\in S$ for some $n$ for, if the orbit of every critical point $w\in K(p)$ is eventually in $T$, it follows that the boundary of every Siegel disk is in $T$, and (as $S$ is simply connected) that every Siegel disk is in $S$.
\end{proof}

\paragraph{Previous upper bounds on $\Sigma$.}
We have noted above that, if $c\in \Ome$ is pseudo-ergodic, then $\Sigma= \spec A_c \subset \overline{W(A_c)} = \overline{\Delta}$, given by \eqref{eq:nr}.
Similarly, the spectrum of $A_c^2$ is contained in the closure of {\em its} numerical range, so that\footnote{Equation \eqref{eq:nr2} is the idea behind higher order numerical ranges; indeed,  where $p$ is the polynomial $p(\lambda) = \lambda^2$, $N_2$ is $\mathrm{Num}(p,A_c)$ in the notation of \cite[p.~278]{LOTS}, so that $N_2$ is a superset of the second order numerical range.}
\begin{equation} \label{eq:nr2}
\Sigma \subset \{\pm \sqrt{z}:z\in \spec(A_c^2)\} \subset
N_2:= \{\pm \sqrt{z}:z\in \overline{W(A_c^2)}\}.
\end{equation}
Hagger \cite{Hagger:NumRange} introduces a new, general method for computing numerical ranges of infinite tridiagonal matrices via the Schur test, which he applies to computing the numerical range of $A_c^2$ (by expressing it as the direct sum of tridiagonal matrices). These calculations show that $\Sigma\subset N_2\subsetneq \overline{\Delta}$; indeed, the calculations in \cite{Hagger:NumRange} imply that $N_2 = \{r\exp(\ri\theta): 0\leq r\leq \rho(\theta), \, 0\leq \theta < 2\pi\}$,
where $\rho\in C(\R)$ is even and periodic with period $\pi/2$, given explicitly on $[-\pi/4,\pi/4]$ by
\begin{equation} \label{eq:nr3}
\rho(\theta) = \left\{\begin{array}{cc}
                     \sqrt{2}, & \pi/6\leq |\theta|\leq \pi/4, \\
                     2/\sqrt{\cos2\theta + \sqrt{3}\ |\sin 2\theta|}, & |\theta| \leq \pi/6.
                                        \end{array}\right.
\end{equation}
By comparison, in polar form,
\begin{eqnarray} \label{eq:nr2a}
W(A_c) = \Delta
 = \{r\re^{\ri \theta}: 0\leq r < 2/(|\cos\theta|+|\sin\theta|),\ 0\leq \theta<2\pi\}.
\end{eqnarray}
Figure \ref{Figure2} includes a visualisation of $\overline{\Delta}$ and $N_2$.

The bound \eqref{eq:nr2}, expressed concretely through \eqref{eq:nr3}, is the sharpest explicit upper bound on $\Sigma$ obtained to date. It implies that $\Sigma$ is not convex since we also know (see \eqref{eq:pi1} below) that $\pm 2$, $\pm 2\ri$, and $\pm 1\pm \ri$ are all in $\Sigma$.

A different family of upper bounds was established in \cite{CCL2} (and see \cite{HengPhD}), expressed in terms of pseudospectra. For a square matrix $A$ of order $n$ and $\epsilon>0$ let $\specn_\epsilon A$ denote the $\epsilon$-pseudospectrum of $A$ (with respect to the 2-norm), i.e., $\specn_\epsilon A := \{\lambda \in \C: \|(A-\lambda I_n)^{-1}\|_2 >\epsilon^{-1}\}$, with the understanding that $\|(A-\lambda I)^{-1}\|_2=\infty$ if $\lambda$ is an eigenvalue, so that $\spec A \subset \spec_\epsilon A$. (Here $\|\cdot\|_2$ is the operator norm of a linear mapping on $\C^n$ equipped with the $2$-norm.) Analogously to \eqref{eq:sigma}, for $\epsilon >0$ and $n\in\N$, let
\begin{equation} \label{eq:sigma2}
\sigma_{n,\epsilon} := \bigcup_{A\in V_n} \specn_\epsilon A,
\end{equation}
which is the union of the pseudospectra of $2^{n-1}$ distinct matrices. Then it is shown in \cite{CCL2} that
\begin{equation} \label{eq:Sigmaconv}
\Sigma^*_n:= \overline{\sigma_{n,\epsilon_n}}\ \searrow\ \Sigma\ \mbox{ as }\ n\to\infty,
\end{equation}
where
$\epsilon_n := 4\sin \theta_n\leq 2\pi/(n+2)$ and $\theta_n$ is the unique solution in $(\pi/(2n+6),\pi/(2n+4)]$ of the equation $2\cos((n+1)\theta) = \cos((n-1)\theta)$.

Clearly, $\{\Sigma^*_n : n\in\N\}$ is a convergent family of upper bounds for $\Sigma$ that is in principle computable; deciding whether  $\lambda\in\Sigma^*_n$ requires only computation of smallest singular values of $n\times n$ matrices (see \cite[(39)]{CCL2}). Explicitly $\Sigma^*_1=2\ovD$, and $\Sigma_n^*$ is plotted for $n=6,12,18$ in \cite{CCL2}. But for these values $\Sigma_n^*\supset \Delta$, and computing $\Sigma_n^*$ for larger $n$ is challenging, requiring computation of the smallest singular value of $2^{n-1}$ matrices of order $n$ to decide whether a particular $\lambda\in \Sigma_n^*$.  Substantial numerical calculations in \cite{CCL2} established that $1.5+0.5\ri\not\in \Sigma^*_{34}$, providing the first proof that $\Sigma$ is a strict subset of $\overline{\Delta}$, this confirmed now by the simple explicit bound \eqref{eq:nr2} and \eqref{eq:nr3}.

\section{Lower Bounds on $\Sigma$ and Symmetries of $\Sigma$ and $\Sigma_\pi$} \label{sec:poly}
Complementing the upper bounds on $\Sigma$ that we have just discussed, lower bounds on $\Sigma$ have been obtained by two methods of argument. The first is that \eqref{eq:spec} tells us that $\spec A_b\subset \Sigma$ for every $b\in \Ome$. In particular this holds in the case when $b$ is periodic, when the spectrum of $A_b$ is given explicitly by Lemmas \ref{lem:spper} and \ref{lem:det}, so that, as observed in the introduction,
$$
\pi_n := \bigcup_{k\in \{\pm 1\}^n} \spec A_k^{\per} \subset \Sigma.
$$
Explicitly \cite[Lemma 2.6]{CCL2}, in particular,
\begin{equation} \label{eq:pi1}
\pi_1 = [-2,2]\cup \ri [-2,2]\ \mbox{ and }\ \pi_2 = \pi_1 \cup \{x\pm \ri x: -1\leq x\leq 1\}.
\end{equation}
In the introduction we have defined $\pi_\infty:= \cup_{n=1}^\infty \pi_n$ and have termed $\Sigma_\pi:=\overline{\pi_\infty}$, also a subset of $\Sigma$ since $\Sigma$ is closed, the periodic part of $\Sigma$. We have also recalled the conjecture of \cite{CCL} that $\Sigma_\pi=\Sigma$. Let
$$
\Pi_n := \bigcup_{m=1}^n \pi_m \subset \pi_\infty \subset \Sigma_\pi \subset \Sigma.
$$
Then it follows from Lemma \ref{eq:setconv} that
\begin{equation} \label{eq:convPin}
\Pi_n \nearrow \Sigma_\pi\ \mbox{ as }\ n\to\infty.
\end{equation}
If, as conjectured, $\Sigma_\pi=\Sigma$, then \eqref{eq:convPin} complements \eqref{eq:Sigmaconv}; together they sandwich $\Sigma$ by convergent sequences of upper ($\Sigma_n^*$) and lower ($\Pi_n$) bounds that can both be computed by calculating eigenvalues of $n\times n$ matrices. Figures \ref{Figure2} and \ref{Figure3}  include visualisations of $\pi_{30}$, indistinguishable by eye from $\Pi_{30}$, but note that the solid appearance of $\pi_{30}$, which is the union of a large but finite number of analytic arcs, is illusory. See \cite{CCL,CCL2} for visualisations of $\pi_{n}$ for a range of $n$, suggestive that the convergence \eqref{eq:convPin} is approximately achieved by $n=30$.

The same method of argument  \eqref{eq:spec} to obtain lower bounds was used in \cite{CCL}, where a special sequence $b\in \Ome$ was constructed with the property that $\spec A_b \supset \ovD$, so that, by \eqref{eq:spec}, $\ovD \subset \Sigma$. The stronger result \eqref{eq:unit}, that this new lower bound on $\Sigma$ is in fact also a subset of $\Sigma_\pi$, was shown in \cite{CWDavies2011}, via
a second method of argument for constructing lower bounds, based on surprising symmetries of $\Sigma$ and $\Sigma_\pi$. We will spell out in a moment these symmetries (one of these described first in \cite{CWDavies2011}, the whole infinite family in \cite{Hagger:symmetries}), which will be both a main subject of study and a main tool for argument in this paper.   But first we note more straightforward but important symmetries. In this lemma and throughout $\overline \lambda$ denotes the complex conjugate of $\lambda\in\C$.

\begin{lem} \label{lem:symeasy} \cite[Lemma 3.4]{CCL2} (and see \cite{HolzOrlZee}, \cite[Lemma 4]{CWDavies2011}). All of $\pi_n$, $\sigma_n$, $\Sigma_\pi$, and $\Sigma$ are invariant with respect to the maps $\lambda\mapsto \ri\lambda$ and $\lambda \mapsto \overline{\lambda}$, and so are invariant under the dihedral symmetry group $D_2$ generated by these two maps.
\end{lem}

To expand on the brief discussion in the introduction, \cite{Hagger:symmetries} proves the existence of an infinite set $\cS$ of monic polynomials of degree $\geq 2$, this set defined constructively in the following theorem, such that the elements $p\in \cS$ are symmetries of $\pi_\infty$ and $\Sigma$ in the sense that \eqref{eq:spsym} below holds.

\begin{thm} \label{thm:cS} \cite{Hagger:symmetries}
Let $\cS$ denote the set of those polynomials $p_k$,  defined by \eqref{eq:pdef2}, with $k=(k_1,...,k_n)\in \{\pm 1\}^n$ for some $n\geq 2$, for which it holds that: (i) $k_{n-1}=-1$ and $k_n=1$; (ii) $p_k=p_{\widehat k}$, where $\widehat k\in \{\pm 1\}^n$ is the vector identical to $k$ but with the last two entries interchanged, so that $\widehat k_{n-1}=1$ and $\widehat k_n = -1$. Then
\begin{equation} \label{eq:spsym}
\Sigma \subset p(\Sigma) \mbox{ and } p^{-1}(\pi_\infty) \subset \pi_\infty,
\end{equation}
for all $p\in \cS$.
\end{thm}

\noindent We will call $\cS$ {\em Hagger's set of polynomial symmetries for $\Sigma$.}

We remark that if $p\in \cS$ then it follows from \eqref{eq:spsym}, by taking closures and recalling that $p$ is continuous, that also
\begin{equation} \label{eq:pply}
p^{-1}(\Sigma_\pi) \subset \Sigma_\pi \mbox{ and } p^{-1}(\intt(\Sigma_\pi)) \subset \intt(\Sigma_\pi).
\end{equation}
We note also that $p^{-1}(\pi_\infty) \subset \pi_\infty$ implies that $\pi_\infty \subset p(\pi_\infty)$, but not vice versa, and that $\Sigma \subset p(\Sigma)$ iff
$$
p^{-1}(\{\lambda\}) \cap \Sigma \neq \emptyset,\ \mbox{ for all } \lambda \in \Sigma.
$$
Further, we note that it was shown earlier in \cite{CWDavies2011} that \eqref{eq:spsym} holds for the particular case $p(\lambda)=\lambda^2$ (this the only element of $\cS$ of degree 2, see Table \ref{table1});  in \cite{CWDavies2011} it was also shown, as an immediate consequence of \eqref{eq:spsym} and Lemma \ref{lem:symeasy}, that
$$
p^{-1}(\Sigma) \subset \Sigma,
$$
for $p(\lambda)=\lambda^2$.
Whether this last inclusion holds in fact for all $p\in \cS$ is an open problem.

Our first result is a much more explicit characterisation of $\cS$.

\begin{prop} \label{prop:eqchar}
The set $\cS$ is given by $\cS=\{p_k:k\in \cK\}$, where $\cK$ consists of those vectors $k=(k_1,...,k_n)\in \{\pm 1\}^n$ with $n\geq 2$, for which:
(i) $k_{n-1} = -1$ and $k_n = 1$; and (ii) $n=2$, or $n\geq 3$ and $k_j = k_{n-j-1}$, for $1 \leq j \leq n-2$, so that $(k_1,..., k_{n-2})$ is a palindrome. Moreover, if $k\in \cK$, then
\begin{equation} \label{p_k}
p_k(\lambda) = \lambda q_{k(1:n-2)}(\lambda).
\end{equation}
\end{prop}
\begin{proof}
It is clear from Theorem \ref{thm:cS} that what we have to prove is that, if $k\in \{\pm 1\}^n$ with $n\geq 2$ and $k_{n-1}=-1$, $k_n=1$, then $p_k=p_{\widehat k}$ if $n=2$ or 3; further, if $n\geq 4$, then $p_k=p_{\widehat k}$ iff $(k_1,...,k_{n-2})$ is a palindrome.

If $k\in \{\pm 1\}^n$ with $n\geq 2$ and $k_{n-1}=-1$, $k_n=1$, then, from \eqref{eq:pk2} and \eqref{eq:qk},
\begin{eqnarray*}
p_k(\lambda) &=& q_{k(1:n-1)}(\lambda) - k_n q_{k(2:n-2)}(\lambda)\\
& = & \lambda q_{k(1:n-2)}(\lambda) + q_{k(1:n-3)}(\lambda) - q_{k(2:n-2)}(\lambda).
\end{eqnarray*}
Thus, if $n=2$ or 3, or $n\geq 4$ and $(k_1,...,k_{n-2})$ is a palindrome, $p_k(\lambda) =  \lambda q_{k(1:n-2)}(\lambda)$ since $q_{k(1:n-3)}(\lambda) = q_{k(2:n-2)}(\lambda)$, this a consequence of the definitions \eqref{eq:ij} in the cases $n=2$ and 3, of Lemma \ref{lem:qksym} in the case $n\geq 4$. Similarly, $p_{\widehat k}(\lambda) =  \lambda q_{k(1:n-2)}(\lambda)$, so that $p_k=p_{\widehat k}$.

Conversely, assume that $k\in \{\pm 1\}^n$ with $n\geq 4$, $k_{n-1}=-1$, $k_n=1$, and $p_k = p_{\widehat{k}}$. To show that $(k_1,...,k_{n-2})$ is a palindrome we need to show that $k_j=k_{n-j-1}$, for $1\leq j\leq(n-2)/2$. Using \eqref{eq:pk2} and then \eqref{eq:qk}, we see that
\begin{eqnarray*}
0 = p_k(\lambda) - p_{\widehat{k}}(\lambda)& =& q_{k(1:n-1)}(\lambda) - q_{\widehat k(1:n-1)}(\lambda) - 2q_{k(2:n-2)}(\lambda)\\
 &=& 2q_{k(1:n-3)}(\lambda) - 2q_{k(2:n-2)}(\lambda).
\end{eqnarray*}
Thus $q_{k(1:n-3)}= q_{k(2:n-2)}$. But, if $q_{k(j:n-j-2)} = q_{k(j+1:n-j-1)}$ and $1\leq j \leq (n-2)/2$, then, applying \eqref{eq:qk},
\begin{eqnarray*}
0&=&q_{k(j:n-j-2)}(\lambda) - q_{k(j+1:n-j-1)}(\lambda)\\
 &=& \lambda q_{k(j+1:n-j-2)}(\lambda)-k_jq_{k(j+2:n-j-2)}(\lambda)\\
 & & \hspace{4ex} - (\lambda q_{k(j+1:n-j-2)}(\lambda)-k_{n-j-1}q_{k(j+1:n-j-3)}(\lambda))\\
 &=& -k_jq_{k(j+2:n-j-2)}(\lambda)+k_{n-j-1}q_{k(j+1:n-j-3)}(\lambda).
\end{eqnarray*}
As this holds for all $\lambda$ and, by Lemma \ref{lem:qkbound} and \eqref{eq:ij}, $q_{k(j+2:n-j-2)}$ and $q_{k(j+1:n-j-3)}$ are both monic polynomials of degree $n-2j-2$, it follows first that $k_{j}=k_{n-j-1}$ and then that $q_{k(j+1:n-j-3)}=q_{k(j+2:n-j-2)}$. Thus that $k_j=k_{n-j-1}$ for $1\leq j\leq(n-2)/2$ follows by induction on $j$.
\end{proof}

The following corollary is immediate from \eqref{p_k} and Lemma \ref{lem:qkbound}.
\begin{cor} \label{cor:small}
Suppose that $n\geq 2$, $k\in \{\pm 1\}^n$, and $p_k\in \cS$. Then, as $\lambda\to 0$, $p_k(\lambda) =\pm\lambda + O(\lambda^3)$ if $n$ is odd, while $p_k(\lambda) = O(\lambda^2)$ if $n$ is even.
\end{cor}

Let us denote by $P_m$ the polynomial $p_k$ when $k$ has length $m\ge 2$, $k_{m-1}=-1$, $k_m=1$, and all other entries are $1$'s. It is convenient also to define $P_1(\lambda)=\lambda$. Clearly, as a consequence of the above proposition, $P_m\in \cS$ for $m\geq 2$ (that these particular polynomials are in $\cS$ was observed earlier in \cite{Hagger:symmetries}). We will write down shortly an explicit formula for $P_m$ in terms of Chebyshev polynomials of the 2nd kind. Recall that $U_n(x)$, the Chebychev polynomial of the 2nd kind of degree $n$, is defined by \cite{AS} $U_0(x):=1$, $U_1(x):=2x$, and $U_{n+1}(x) := 2xU_n(x)-U_{n-1}(x)$, for $n\in\N$.
\begin{lem} \label{lem:cheb}
For $m\in \N$, $P_m(\lambda) = \lambda U_{m-1}(\lambda/2)$.
\end{lem}
\begin{proof}
This follows easily by induction from \eqref{p_k} and \eqref{eq:qk}.
\end{proof}
We note that, using the standard trigonometric representations for the Chebychev polynomials \cite{AS}, for $m\in \N$,
\begin{equation} \label{eq:theta}
P_m(2 \cos \theta) = 2 \cos\theta U_{m-1}(\cos\theta) = 2\cot\theta \sin m\theta =: r_m(\theta).
\end{equation}
A similar representation in terms of hyperbolic functions can be given for the polynomial $p_k$ when $k$ has length $2m-1$ and $k_j = (-1)^j$; we denote this polynomial by $Q_m$. Clearly, for $m\geq 2$, $Q_m\in \cS$ by Proposition \ref{prop:eqchar}, and $Q_m$ is an odd function by Lemma \ref{lem:det}.  The proof of the following lemma, like that of Lemma \ref{lem:cheb}, is a straightforward induction that we leave to the reader.

\begin{lem} \label{lem:cheb2}
$Q_1(\lambda) = \lambda$, $Q_2(\lambda) = \lambda^3 + \lambda$, and $Q_{m+1}(\lambda) = \lambda^2 Q_m(\lambda) + Q_{m-1}(\lambda)$, for $m\geq 2$. Moreover, for $m\in \N$ and $\theta\geq 0$,
\[Q_m\left(\sqrt{2\sinh \theta}\right) = \begin{cases} \sqrt{2\sinh\theta}\ \displaystyle{\frac{\sinh(m\theta) + \cosh((m-1)\theta)}{\cosh\theta}}, & \text{if $m$ is even,} \\ \sqrt{2\sinh \theta}\ \displaystyle{\frac{\cosh(m\theta) + \sinh((m-1)\theta)}{\cosh \theta}}, & \text{if $m$ is odd.}\end{cases}\]
\end{lem}

The following lemma leads, using Lemmas \ref{lem:cheb} and \ref{lem:cheb2}, to explicit formulae for other polynomials in $\cS$. For example, if  $P_m^*$ denotes the polynomial $p_k$ when $k$ has length $m\ge 2$, $k_{m-1}=-1$, $k_m=1$, and all other entries are $-1$'s, then, by Lemmas \ref{lem:cheb} and \ref{lem:minus},
\begin{equation} \label{eq:PM*}
P_m^*(\lambda) = \ri^{-m} P_m(\ri \lambda) = \ri^{1-m}\lambda U_{m-1}(\ri \lambda/2).
\end{equation}
\begin{lem} \label{lem:minus} If $k\in \{\pm 1\}^n$ and $p_k\in \cS$, then $p_{-k}\in \cS$ and $p_{-k}(\lambda)=\ri^{-n} p_k(\ri\lambda)$.
\end{lem}
\begin{proof}
Suppose that $k\in \{\pm 1\}^n$ and $p_k\in \cS$. If $n=2$, then $\widehat{k} = -k$ and $p_{-k}=p_{\widehat k} = p_k\in \cS$. If $n\geq 3$, defining $\ell\in \{\pm 1\}^n$ by $\ell_{n-1}=-1$, $\ell_n=1$, and $\ell_j=-k_j$, for $j=1,...,n-2$, $p_\ell\in \cS$ by Proposition \ref{prop:eqchar}, so that $p_{-k}=p_{\widehat \ell}=p_\ell\in \cS$. That $p_{-k}(\lambda)=\ri^{-n} p_k(\ri\lambda)$ comes from Lemma \ref{lem:pksym}.
\end{proof}

We note that Proposition \ref{prop:eqchar} implies that there are precisely $2^{\left\lceil\frac{n}{2}\right\rceil-1}$ vectors of length $n$ in $\cK$, so that there are between 1 and $2^{\left\lceil\frac{n}{2}\right\rceil-1}$ polynomials of degree $n$ in $\cS$, as conjectured in \cite{Hagger:symmetries}. Note, however, that there may be more than one $k \in \cK$ that induce the same polynomial $p_k\in \cS$. For example, $p_k(\lambda) = \lambda^6-\lambda^2$ for $k=(-1,1,1,-1,-1,1)$, and, defining $\ell=(1,-1,-1,1,-1,1)$ and using Lemma \ref{lem:minus}, also
$p_\ell(\lambda) = p_{\widehat \ell}(\lambda)=p_{-k}(\lambda) = -p_k(\ri \lambda) = \lambda^6-\lambda^2$.
In Table \ref{table1} (cf.~\cite{Hagger:symmetries}) we tabulate all the polynomials in $\cS$ of degree $\leq 6$.

\begin{table}
  \centering
  \begin{tabular}{|r|l|}
    \hline
    $k$ & $p_k(\lambda)$ \\
    \hline
    $(-1,1)$ & $\lambda^2= P_2(\lambda)$\\
    \hline
    $(1,-1,1)$ & $\lambda^3-\lambda= P_3(\lambda)$\\
    \hline
    $(-1,-1,1)$ & $\lambda^3+\lambda= Q_2(\lambda)=P_3^*(\lambda)$\\
    \hline
        $(1,1,-1,1)$ & $\lambda^4-2\lambda^2= P_4(\lambda)$\\
    \hline
            $(-1,-1,-1,1)$ & $\lambda^4+2\lambda^2=P_4^*(\lambda)$\\
    \hline
           $(1,1,1,-1,1)$ & $\lambda^5-3\lambda^3+\lambda=P_5(\lambda)$\\
    \hline
           $(1,-1,1,-1,1)$ & $\lambda^5-\lambda^3+\lambda= -\ri Q_3(\ri\lambda)$\\
    \hline
           $(-1,1,-1,-1,1)$ & $\lambda^5+\lambda^3+\lambda=Q_3(\lambda)$\\
    \hline
           $(-1,-1,-1,-1,1)$ & $\lambda^5+3\lambda^3+\lambda=P_5^*(\lambda)$\\
    \hline
              $(1,1,1,1,-1,1)$ & $\lambda^6-4\lambda^4+3\lambda^2=P_6(\lambda)$\\
    \hline
                  $(1,-1,-1,1,-1,1)$ & $\lambda^6-\lambda^2=P_3(P_2(\lambda))$\\
    \hline
              $(-1,-1,-1,-1,-1,1)$ & $\lambda^6+4\lambda^4+3\lambda^2=P_6^*(\lambda)$\\
    \hline
  \end{tabular}
  \caption{The elements $p_k\in \cS$ of degree $\leq 6$.}\label{table1}
\end{table}

If $p,q\in \cS$, so that $p$ and $q$ are polynomial symmetries of $\Sigma$ in the sense that \eqref{eq:spsym} holds, then also the composition $p\circ q$ is a polynomial symmetry of $\Sigma$ in the same sense. But note from Table \ref{table1} that, while $P_3\circ P_2\in \cS$, none of $P_2\circ P_2$, $P_2\circ P_2\circ P_2$, $Q_2\circ P_2$, $P_2\circ P_3$, or $P_2\circ Q_2$ are in $\cS$. Thus $\cS$ does not contain all polynomial symmetries of $\Sigma$, but whether there are polynomial symmetries that are not either in $\cS$ or else compositions of elements of $\cS$ is an open question.

We finish this section by showing in subsection \ref{subsec:a} the surprising result that $\cS$ is large enough that we can reconstruct the whole of $\Sigma_\pi$ from the polynomials $p_k\in \cS$. This result will in turn be key to the proof of our main theorem in Section \ref{sec:interior}. Then in subsection \ref{subsec:b} we use that \eqref{eq:spsym} holds for the polynomials in Table \ref{table1} to obtain new explicit lower bounds on $\Sigma_\pi$, including that $1.1\ovD \subset \Sigma_\pi\subset \Sigma$.

\subsection{Connecting eigenvalues of finite matrices and polynomial symmetries of $\Sigma$} \label{subsec:a}

Recall from Proposition \ref{prop:eqchar} that $\cS=\{p_k:k\in \cK\}$, let $K := \cup_{n\in \N} \{\pm 1\}^n$, and define
\begin{equation} \label{eq:piinfS}
\pi_\infty^{\cS} := \bigcup_{k\in \cK} \spec A^{\per}_k \subset \pi_\infty = \bigcup_{k\in K} \spec A^{\per}_k.
\end{equation}
The following result seems rather surprising, given that $\cK$ is much smaller than $K$ in the sense that there are precisely $2^{\left\lceil\frac{n}{2}\right\rceil-1}$ vectors of length $n$ in $\cK$ but $2^n$ in $K$.

\begin{thm} \label{thm:dense}
$\sigma_\infty \subset \pi_\infty^{\cS}$, so that $\pi_\infty^{\cS}$ is dense in $\Sigma_\pi$ and
$
\Sigma_\pi := \overline{\pi_\infty} = \overline{\sigma_\infty} = \overline{\pi_\infty^{\cS}}.
$
\end{thm}
\begin{proof} We will show in Proposition \ref{prop:inc} below that, for $n\geq2$,
$$
\sigma_n \subset \pi_{2n+2}^\cS := \bigcup_{k\in \cK_{2n+2}} \spec A^{\per}_k \subset \pi_{2n+2},
$$
where, for $m\geq 2$, $\cK_m$ denotes the set of those vectors in $\cK$ that have length $m$. Since also $\sigma_1=\{0\}\subset \pi_m^{\cS}$, for every $m\in \N$ (e.g.~\cite[Lemma 2.10]{CCL2}), this implies that $\sigma_\infty \subset \pi_\infty^{\cS}$, which implies that $\pi_\infty^{\cS}$ is dense in $\Sigma_\pi$ since  $\sigma_\infty$ is dense in $\pi_\infty$ \cite[Theorem 1]{Hagger:dense}, and the result follows.
\end{proof}

The key step in the proof of the above theorem is the following refinement of Theorem 4.1 in \cite{CCL2}, which uses our new characterisation, Proposition \ref{prop:eqchar}, of $\cS$.

\begin{prop} \label{prop:inc} Suppose $a,b,c,d\in \{\pm 1\}$ and $k\in \{\pm 1\}^n$, for some $n\geq 2$, and let $\widetilde k := (k_1,...,k_{n-1})$. Then
\begin{equation} \label{eq:elldef}
\spec A_k^{(n)} \subset \spec A_\ell^{\per},\ \mbox{ for }\ \ell =({\widetilde k}^\prime,a,b,\widetilde k,c,d)\in \{\pm 1\}^{2n+2},
\end{equation}
where ${\widetilde k}^\prime = \widetilde k J_{n-1}  = (k_{n-1},...,k_1)$. Further, $\spec A_k^{(n)} \subset \pi_{2n+2}^\cS$.
\end{prop}
\begin{proof} The proof modifies \cite[Theorem 4.1]{CCL2} where the same result is proved for the special case that $a=c=-1$, $b=d=1$. Following that proof, suppose that $\lambda$ is an eigenvalue of $A_k^{(n)}$ with corresponding eigenvector $x$, let $\widehat x:= J_n x$ and $\widehat A_k^{(n)}:= J_n  A_k^{(n)} J_n$, so that $\widehat A_k^{(n)} \widehat x = \lambda \widehat x$, and set
\begin{equation}\nonumber
B\ :=\
\left(\begin{array}{ccccccc}
\ddots&
\begin{array}{ccc}1&& \end{array}\\
\cline{2-2}
\begin{array}{c}d\\ \\ \\ \end{array}&
\multicolumn{1}{|c|}{\widehat A_k^{(n)}}&
\begin{array}{c} \\ \\-a \end{array}\\
\cline{2-2}
&\begin{array}{ccr}&&-1\end{array}
& \boxed{0} &
\begin{array}{lcc}1&&\end{array}\\
\cline{4-4}
& & \begin{array}{c}b\\ \\ \\ \end{array}&
\multicolumn{1}{|c|}{{A_k^{(n)}}}&
\begin{array}{c} \\ \\-c \end{array}\\
\cline{4-4}
&&&\begin{array}{ccr}&&-1\end{array}
& 0 & \begin{array}{lcc}1&&\end{array}\\
\cline{6-6}
& && & \begin{array}{c}d\\ \\ \\ \end{array}&
\multicolumn{1}{|c|}{\widehat A_k^{(n)}}&
\begin{array}{c} \\ \\-a \end{array}\\
\cline{6-6}
&&&&&
\begin{array}{ccc}&&-1\end{array}&
\ddots
\end{array}\right),
\end{equation}
where $\boxed{0}$ marks the entry at position $(-1,-1)$. Then $B$ is a bi-infinite tridiagonal matrix with zeros on the main diagonal, and with each of the first sub- and super-diagonals a vector in $\Omega$ that is periodic with period $2n+2$. Define $\tilde x\in \ell^\infty$, the space of bounded, complex-valued sequences $\phi:\Z\to \C$, by
$$
\tilde x := (...,0,\widehat x^T,\boxed{0},x^T,0,\widehat x^T,0,x^T,...)^T,
$$
where $\boxed{0}$ marks the entry $\tilde x_{-1}$. Then it is easy to see that $B\tilde x =\lambda \tilde x$, so that $\lambda \in \spec B$.\footnote{Clearly $\lambda$ is in the spectrum of $B$ as an operator on $\ell^\infty$, but this is the same as the $\ell^2$-spectrum from general results on band operators, e.g.~\cite{LiBook}.}  Further, by a simple gauge transformation \cite[Lemma 3.2]{CCL2}, $\spec B = \spec A_\ell^{\per}$, where $\ell$ is given by \eqref{eq:elldef}.

We have shown that $\spec A_k^{(n)} \subset \spec A_\ell^{\per}$. But, choosing in particular $a=b$ and $c=-1$, $d=1$, we see from Proposition \ref{prop:eqchar} that $\ell\in \cK_{2n+2}$, so that $\spec A_k^{(n)} \subset \spec A_\ell^{\per} \subset\pi_{2n+2}^\cS$.
\end{proof}

\begin{rem} \label{rem:one}
Proposition \ref{prop:inc} implies that $\spec A_k^{(n)} \subset \spec A_\ell^{\per}$ for 16 different vectors $\ell\in \{\pm 1\}^n$, corresponding to different choices of $a,b,c,d\in \{\pm 1\}$. Some of these vectors correspond to the same polynomial $p_\ell$ and hence, by \eqref{eq:per}, to the same spectrum $\spec A_\ell^{\per}$. In particular, if $a=b=\pm 1$, then the choices $c=-d=1$ and $c=-d=-1$ lead to the same polynomial by Proposition \ref{prop:eqchar} and the definition of $\cS$. But, if $a\neq b$, again by Proposition \ref{prop:eqchar} and the definition of $\cS$, the choices $c=-d=1$ and $c=-d=-1$ must lead to different polynomials, and neither of these polynomials can be in $\cS$. On the other hand, as observed already in the proof of Proposition \ref{prop:inc}, the choices $a=b$ and $c=-d$ lead to an $\ell\in \cK$ and so to a polynomial $p_\ell\in \cS$. Thus Proposition \ref{prop:inc} implies that there are at least three distinct polynomials $p_\ell$ such that $\spec A_k^{(n)} \subset \spec A_\ell^{\per}$. Figure \ref{Figure1} plots $\spec A_k^{(n)}$ and $\spec A_\ell^{\per}$, for the vectors defined by \eqref{eq:elldef}, in the case that $n=3$ and $k_1=k_2=1$. For other plots of the spectra of finite and periodic Feinberg-Zee matrices, and the interrelation of these spectra, see \cite{HolzOrlZee,CCL,CCL2}.
\end{rem}

\begin{figure}
\begin{center}
\includegraphics[scale = 0.7, trim = 0 1cm 0 0]{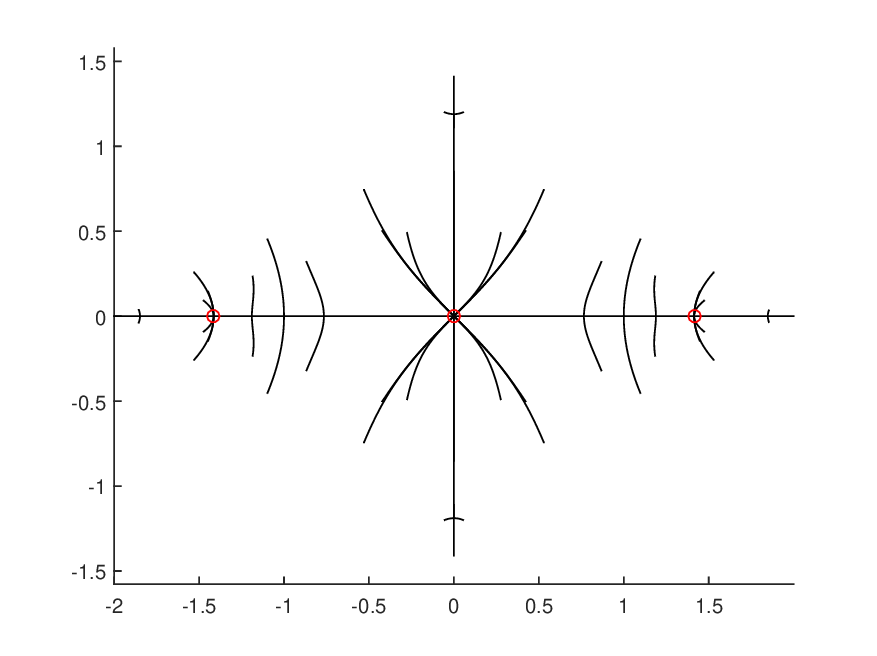}

\caption{An illustration of Proposition \ref{prop:inc} in the case $n = 3$, $k_1 = k_2 = 1$. The red circles indicate the eigenvalues, 0 and $\pm \sqrt{2}$, of $A_k^{(3)}$. The black lines are the spectra of $A_\ell^{\per}$, for the different choices of $\ell$ defined by \eqref{eq:elldef}. In this case there are $7$ distinct polynomials $p_\ell$ and $7$ associated distinct spectra $\spec A_\ell^{\per}$, each of which contains the eigenvalues of $A_k^{(3)}$. One cannot see all the spectra as separate curves because some of them overlap.} \label{Figure1}
\end{center}
\end{figure}

\subsection{Explicit lower bounds for $\Sigma_\pi$} \label{subsec:b}

As noted in \cite{Hagger:symmetries}, that the polynomials $p\in \cS$ satisfy \eqref{eq:spsym} gives us a tool to compute explicit lower bounds on $\Sigma_\pi$. Indeed \cite{Hagger:symmetries} shows visualisations of $p^{-n}(\D)$, for several $p\in \cS$ and $n\in N$, and visualisations of unions of $p^{-n}(\D)$ for $p$ varying over some finite subset of $\cS$. Since, by \eqref{eq:unit}, $\D\subset \Sigma_\pi$, it follows from \eqref{eq:spsym} that all these sets are subsets of $\Sigma_\pi$.

The study in \cite{Hagger:symmetries} contains visualisations of subsets of $\Sigma_\pi$ as just described, but no associated analytical calculations. Complementing the study in \cite{Hagger:symmetries} we make explicit calculations in this section that illustrate the use of the polynomial symmetries to compute explicit formulae for regions of the complex plane that are subsets of $\Sigma_\pi$, adding to the already known fact \eqref{eq:unit} that $\D\subset \Sigma_\pi$.

Our first lemma and corollary give explicit formulae for $p^{-1}(\D)$ when $p(\lambda)=P_3(\lambda) = \lambda^3-\lambda$ and when $p(\lambda) = Q_2(\lambda)=\lambda^3+\lambda$. These formulae are expressed in terms of the function $s\in C^\infty[-1,1]$, where, for $-1\leq t\leq 1$, $s(t)>0$ is defined as the largest positive solution of
$$
f(s):= s^3 - 2ts^2 + s = 1.
$$
If $-1\leq t<1/2$, this is the unique solution in $(0,1)$ (on which interval $f^\prime(s)\geq 0$), while if $1/2\leq t< 1$ it is the unique solution in $[1,2)$ (on which interval $f^\prime(s)\geq 0$). Explicitly, for $-1\leq t\leq 1$,
$$
s^\prime(t) = \frac{2(s(t))^3}{2-s(t) + (s(t))^3} >0,
$$
and $s(1/2)=1$, $s(1) \approx 1.75488$.

\begin{lem} \label{lem:degree3}
If $p(\lambda):= P_3(\lambda)=\lambda^3-\lambda$, then
$$
p^{-1}(\D) = E:=\{r\re^{\ri\theta}: 0\leq r < S(\theta), 0\leq\theta< 2\pi\},
$$
where $S(\theta) := \sqrt{s(\cos2\theta)}$, for $\theta\in \R$. $S\in C^\infty(\R)$ and is even and periodic with period $\pi$. In the interval $[0,\pi]$ the only stationary points of $S$ are global maxima at $0$ and $\pi$, with $S(0)=S(\pi) =\sqrt{s(1)}\approx  1.32472$,
and a global minimum at $\pi/2$. Further, $S$ is strictly decreasing on $[0,\pi/2]$ and $S(\theta)\geq 1$ in $[0,\pi]$ iff $0\leq \theta \leq \pi/6$ or $5\pi/6\leq \theta\leq \pi$, with equality iff $\theta = \pi/6$ or $5\pi/6$.
\end{lem}
\begin{proof}
If $\lambda = r\re^{\ri\theta}$ (with $r\geq 0$, $\theta\in\R$), then
$|p(\lambda)|^2 = |\lambda^3-\lambda|^2 = |r^3\re^{2\ri\theta}-r|^2 = r^6-2r^4\cos(2\theta) + r^2$. It is straightforward to show that
$|p(\lambda)|<1$ iff  $0\leq r< S(\theta)$.
The properties of $S$ claimed follow easily from the properties of $s$ stated above.
\end{proof}

\begin{cor} \label{cor:degree3}
If $p(\lambda):= Q_2(\lambda)=\lambda^3+\lambda$, then
$$
p^{-1}(\D) = \ri E =\{r\re^{\ri\theta}: 0\leq r < S(\theta-\pi/2), 0\leq\theta< 2\pi\}.
$$
\end{cor}
\begin{proof} This is clear from Lemma \ref{lem:degree3} and the observation that $Q_2(\lambda) = \ri P_3(\ri\lambda)$.
\end{proof}

Since $P_3, Q_2\in \cS$, it follows from the above lemma and corollary and \eqref{eq:pply} that
$
E\cup \ri E \subset \intt(\Sigma_\pi)$.
But this implies by \eqref{eq:pply}, since $P_2(\lambda)=\lambda^2$ is also in $\cS$, that also
$
\sqrt{\ri E}  \subset \intt(\Sigma_\pi)$,
where, for $S\subset \C$, $\sqrt{S}:= \{\lambda \in \C: \lambda^2\in S\}$.
In particular,
\begin{eqnarray} \label{eq:W1}
W_1 := \{r\re^{\ri\theta}: 0\leq r < S(\theta), -\pi/6\leq\theta\leq \pi/6\}\subset E\subset \intt(\Sigma_\pi)\ \; \mbox{ and}\\ \label{eq:W2}
W_2:= \{r\re^{\ri\theta}: 0\leq r < \sqrt{S(2\theta-\pi/2)}, \pi/6\leq\theta\leq \pi/3\}\subset \sqrt{\ri E} \subset \intt(\Sigma_\pi).
\end{eqnarray}
It is easy to check that
\begin{equation} \label{eq:union}
\bigcup_{m=0}^3 \ri^m (W_1\cup W_2) = E\cup \ri E \cup \sqrt{\ri E} \subset \intt(\Sigma_\pi).
\end{equation}

Next note from Table \ref{table1} that $p\in \cS$ where $p(\lambda) := \lambda^5 - \lambda^3+\lambda$ factorises as
$$
p(\lambda) = \lambda(\lambda - \re^{\ri \pi/6})(\lambda + \re^{\ri \pi/6})(\lambda - \re^{-\ri \pi/6})(\lambda + \re^{-\ri \pi/6}).
$$
Thus, for $\lambda= \exp(\pm \ri \pi/6)+w$ with $|w| \leq \epsilon$,
$$
|p(\lambda)|\leq (1+\epsilon)\epsilon(2+\epsilon)(2\sin(\pi/6)+\epsilon)(2\cos(\pi/6)+\epsilon)=\epsilon(1+\epsilon)^2(\sqrt{3}+\epsilon)(2+\epsilon)=:g(\epsilon).
$$
Let $\eta\approx 0.174744$ be the unique positive solution of $g(\epsilon)=1$.  Clearly $|p(\lambda)| < 1$ if $\lambda = \exp(\pm\ri\pi/6)+w$, with $|w|<\eta$, so that
\begin{equation} \label{eq:iun}
\exp(\pm\ri \pi/6) + \eta \D \subset p^{-1}(\D) \subset \intt(\Sigma_\pi).
\end{equation}

\begin{figure}
\begin{center}
\includegraphics[scale = 0.7, trim = 0 1cm 0 0]{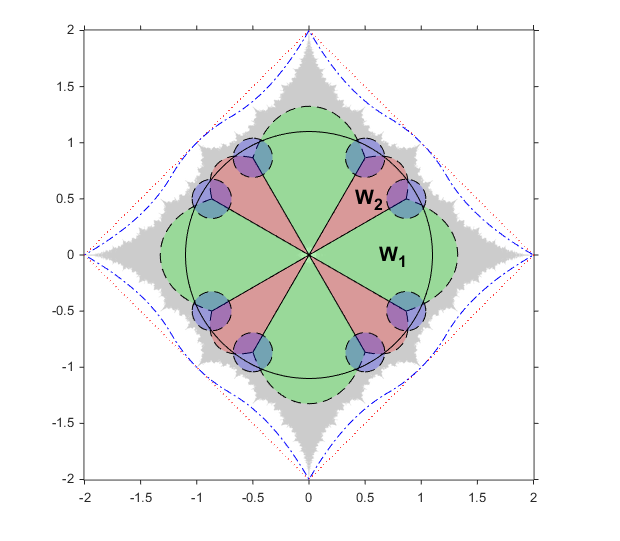}

\caption{A plot showing $W_1$ (green), $W_2$ (red), $\re^{\pm\ri\pi/6}+\eta \D$ (blue), and their rotations by multiples of $\pi/2$. The union of the green, red, and blue regions is $W\subset \Sigma_\pi$, defined by \eqref{eq:bb}. $W$ contains $1.1\ovD$, indicated by the black circle, see Proposition \ref{prop:betterbounds}. In the background in grey one can see $\pi_{30}\subset \Sigma_\pi$. The dotted and dashed-dotted curves are the boundaries of $\overline{\Delta}$ and $N_2$, respectively, defined by \eqref{eq:nr} and \eqref{eq:nr2}, with $\overline{\Delta} \supset N_2\supset \Sigma \supset \Sigma_\pi$.} \label{Figure2}
\end{center}
\end{figure}

We have shown most of the following proposition that extends to a region $W\supset 1.1\ovD$ (illustrated in Figure \ref{Figure2}) the part of the complex plane that is known to consist of interior points of $\Sigma_\pi$, making explicit implications of the polynomial symmetries $\cS$ of $\Sigma$. Before \cite{Hagger:symmetries} and the current paper the most that was known explicitly was that $\D\subset \intt(\Sigma_\pi)$.

\begin{prop} \label{prop:betterbounds}
\begin{equation} \label{eq:bb}
1.1\ovD \subset W := \bigcup_{m=0}^3 \ri^m \left(W_1\cup W_2 \cup (\re^{\ri\pi/6}+\eta \D) \cup (\re^{-\ri\pi/6}+\eta \D)\right)\subset \intt(\Sigma_\pi).
\end{equation}
\end{prop}
\begin{proof}
That $W\subset \intt(\Sigma_\pi)$ is \eqref{eq:union} combined with \eqref{eq:iun} and Lemma \ref{lem:symeasy}. It is easy to see that $W$ is invariant with respect to the maps $\lambda\mapsto \ri\lambda$ and $\lambda \mapsto \overline{\lambda}$, and so is invariant under the dihedral symmetry group $D_2$ generated by these two maps. Thus to complete the proof it is sufficient to show that $z:=r\re^{\ri \theta}\in W$ for $0\leq r\leq 1.1$ and $0\leq \theta\leq \pi/4$.

Now since, by Lemma \ref{lem:degree3}, $S(\theta)\geq 1$ for $-\pi/6\leq \theta\leq \pi/6$, $z\in W_1\cup W_2$ for $0\leq r\leq 1$, $0\leq \theta \leq \pi/4$. Further, since, by Lemma \ref{lem:degree3}, $S$ is even and is increasing on $[-\pi/6,0]$, $S(\theta)\geq S(\pi/8) = \sqrt{s(\sqrt{2}/2)\,}$, for $0\leq \theta \leq \pi/8$, and $\sqrt{S(2\theta-\pi/2)} \geq \sqrt{S(\pi/12)}=\left(s(\sqrt{3}/2)\right)^{1/4}$, for $5\pi/24\leq \theta \leq \pi/4$. But $\sqrt{3}>1.73$, so that $f(1.5) < 0.9825$ for $t=\sqrt{3}/2$, which implies that $s(\sqrt{3}/2)>1.5$, so that
$\left(s(\sqrt{3}/2)\right)^{1/4}>1.1$. Similarly, $\sqrt{2}/2>0.705$ and $s(0.705) = 1.25$, so that $\sqrt{s(\sqrt{2}/2)\,}>\sqrt{1.25}>1.1$. Thus $z\in W_1$ for $0\leq r\leq 1.1$ and $0\leq \theta\leq \pi/8$, while $z\in W_2$ for $0\leq r\leq 1.1$ and $5\pi/24\leq \theta\leq \pi/4$.

To conclude that $1.1\ovD \subset W$, it remains to show that $z\in W$ for $1 \leq r\leq 1.1$ and $|\pi/6-\theta|\leq \pi/24$. But it is easy to check that, for these ranges of $r$ and $\theta$, $z\in \exp(\ri \pi/6)+ \eta \D$ provided $\cos(\pi/24) + \sqrt{\cos^2(\pi/24) + \eta^2-1} > 1.1$. But this last inequality holds since $\cos(\pi/24)= \frac{1}{2}\sqrt{2+\sqrt{2+\sqrt{3}}}>0.991$ and $g(0.174)<1$ so that $\eta > 0.174$.
\end{proof}

\section{Interior points of $\Sigma_\pi$} \label{sec:interior}

We have just, in Proposition \ref{prop:betterbounds}, extended to a region $W\supset 1.1\ovD$ the part of the complex plane that is known to consist of interior points of $\Sigma_\pi$. In this section we explore the relationship between $\Sigma_\pi$ and its interior further. We show first of all, using \eqref{eq:pply} and that $P_n\in \cS$ for every $n\geq 2$, that $[0,2)\subset \intt(\Sigma_\pi)$. Next we use this result to show that, for every $n\geq 2$, all but finitely many points in $\pi^\cS_n$ are interior points of $\Sigma_\pi$. Finally, we prove, using Theorem \ref{thm:dense}, that $\Sigma_\pi$ is the closure of its interior. If indeed it can be shown, as conjectured in \cite{CCL}, that $\Sigma_\pi=\Sigma$, then the result will imply the truth of another conjecture in \cite{CCL}, that $\Sigma$ is the closure of its interior.

Our technique for establishing that $[0,2)\subset \intt(\Sigma_\pi)$ will be to use that $P_m^{-1}((-1,1)) \subset \intt(\Sigma_\pi)$, for every $m\geq 3$, this a particular instance of \eqref{eq:pply}. This requires first a study of the real solutions of the equations $P_m(\lambda)=\pm 1$ and their interlacing properties, which we now undertake.

From \eqref{eq:theta},
$$
P_m(2) = 2m,\ P_m(2\cos(\pi/m)) = 0,\ \mbox{ and }\ P_m(2\cos(3\pi/(2m))) = -2 \cot(3\pi/(2m)).
$$
This implies that the equation $P_m(\lambda)=1$ has a solution in $(2\cos(\pi/m),2)$. Let $\lambda_m^+$ denote the largest solution in this interval. Further, if $m\ge 5$ then $P_m(\lambda)=-1$ has a solution in $(2\cos(3\pi/(2m)),2)$ since $-2 \cot(3\pi/(2m))<-1$. For $m\ge 4$ let $\lambda_m^-$ denote the largest solution to $P_m(\lambda)=-1$ in $(0,2)$, which is in the interval $(2\cos(3\pi/(2m)),2)$ if $m\ge 5$, while an explicit calculation gives that $\lambda_4^-=1$.

Throughout the following calculations we use the notation $r_n(\theta)$ from \eqref{eq:theta}.
\begin{lem} \label{lem:inc}
For $m\ge 4$ it holds that $P_m$ is strictly increasing on $(\lambda_m^-,2)$, that $\lambda_m^-<\lambda_m^+$, and that $-1<P_m(\lambda)<1$ for $\lambda_m^-<\lambda<\lambda_m^+$.
\end{lem}
\begin{proof}
Explicitly, $P_4(\lambda) = \lambda U_3(\lambda/2)=\lambda^4-2\lambda^2$, so that $P_4^\prime(\lambda)=4(\lambda^3-1)$ and these claims are clear for $m=4$.

Suppose now that $m\ge 5$. It follows by induction that, for $n\geq 3$, $r_n(\theta)$ is strictly decreasing on $(0,\pi/n+\pi/n^2)$. For $r_3(\theta)= 2\cos\theta(4\cos^2\theta-1)$ is strictly decreasing on $(0,4\pi/9)\subset (0,\pi/2)$ and, if this statement is true for some $n\geq 3$, then
$$
r_{n+1}(\theta) = 2\cot\theta \sin((n+1)\theta) = \cos\theta(r_n(\theta) + 2\cos n\theta)
$$
is strictly decreasing on $(0,\pi/(n+1)+\pi/(n+1)^2)\subset (0,\pi/n)$. Further,
\begin{eqnarray*}
r_m(\pi/m+\pi/m^2) &=& -2\cos(\pi/m+\pi/m^2)\sin(\pi/m)/\sin(\pi/m+\pi/m^2)\\
& <&-\ \frac{2m\cos(\pi/5+\pi/25)}{m+1}<-10\cos(6\pi/25)/6 < -1,
\end{eqnarray*}
 since $\sin a/\sin b > a/b$ for $0<a<b<\pi$. As $r_m(\theta)=P_m(2\cos\theta)$, these observations imply that, on $(2\cos(\pi/m+\pi/m^2),2)$, $P_m$ is strictly increasing, and that $\lambda_m^->2\cos(\pi/m+\pi/m^2)$. Thus $\lambda_m^+>\lambda_m^-$ follows from the definitions of $\lambda_m^\pm$, and $-1<P_m(\lambda)<1$ for $\lambda_m^-<\lambda<\lambda_m^+$.
\end{proof}

 As observed above it follows from \eqref{eq:pply} that $P_m^{-1}((-1,1))\subset \intt(\Sigma_\pi)$, for $m\geq 2$. Combining this observation with Lemma \ref{lem:inc} and the fact that $P_3(\lambda)=\lambda^3-\lambda \in (-1,1)$ if $-\lambda_3^+<\lambda<\lambda_3^+$, we obtain the following corollary.

\begin{cor} \label{cor:intpts}
For $m=4,5,...$, $(\lambda_m^-,\lambda_m^+)\subset \intt(\Sigma_\pi)$. Also, $(-\lambda_3^+,\lambda_3^+)\subset \intt(\Sigma_\pi)$.
\end{cor}

By definition of $\lambda_m^+$, $\lambda_m^+\to 2$ as $m\to\infty$. Thus the above corollary and the following lemma together imply that $[0,2)\subset \intt(\Sigma_\pi)$.

\begin{lem} \label{lem:inter}
For $m=3,4,...$, $\lambda_{m+1}^-<\lambda_m^+<\lambda_{m+1}^+$.
\end{lem}
\begin{proof} Since $P_3(\lambda) = \lambda^3-\lambda$ and $P_4(\lambda) = \lambda^4-2\lambda^2$, it is easy to see that $1=\lambda_4^-<\lambda_3^+ < \lambda_4^+$, so that the claimed result holds for $m=3$.

To see the result for $m\ge 4$ we will show, equivalently, that $r_{m+1}^->r_m^+>r_{m+1}^+$, for $m=4,5,...$. Here $r_n^+$, for $n\in\N$, is the smallest solution of $r_n(\theta)=1$ in $(0,\pi/n)$, so that $\lambda_n^+=2\cos r_n^+$, while $r_n^-$, for $n=4,5,...$, denotes the smallest solution of $r_n(\theta)=-1$ in $(0,\pi/2)$, so that $\lambda_n^-=2\cos r_n^-$.

We have shown in the proof of Lemma \ref{lem:inc} that, for $m\geq 3$, $r_m(\theta)$ is strictly decreasing on $(0,\pi/m+\pi/m^2)$, and that $r_m(\pi/m+\pi/m^2) < -1$, for $m\ge 5$, while $r_m(\pi/m)=0$ so that
$$
\frac{\pi}{m} < r_m^- < \frac{\pi}{m} + \frac{\pi}{m^2}, \quad \mbox{for } m\ge 5.
$$
Similarly, for $m\ge 2$,
\begin{eqnarray*}
r_m(\pi/m-\pi/m^2) &=& 2\cos(\pi/m-\pi/m^2)\sin(\pi/m)/\sin(\pi/m-\pi/m^2)\\
&>& 2\cos(\pi/2-\pi/4)=\sqrt{2} > 1,
\end{eqnarray*}
so that
$$
\frac{\pi}{m} - \frac{\pi}{m^2}< r_m^+ < \frac{\pi}{m}, \quad \mbox{for } m\ge 2.
$$
These inequalities imply that, for $m\ge 4$, $r_{m+1}^-\in (0,\pi/m+\pi/m^2)$, and since $r_m$ is strictly decreasing on this interval and
\begin{eqnarray*}
r_m(r_{m+1}^-) &=& 2\cot r_{m+1}^- \sin((m+1)r_{m+1}^- - r_{m+1}^-)\\
&  =& \cos r_{m+1}^-(-1-2\cos((m+1)r_{m+1}^-)) < \cos r_{m+1}^- < 1,
\end{eqnarray*}
it follows that $r_m^+<r_{m+1}^- <\pi/(m+1)+\pi/(m+1)^2$. Since also $r_{m+1}$ is strictly decreasing on $(0,\pi/(m+1)+\pi/(m+1)^2)$ and
$$
r_{m+1}(r_m^+) = 2\cot r_m^+ \sin(mr_m^+ + r_m^+) = \cos r_m^+ (1+2\cos(mr_m^+)) < \cos r_m^+<1,
$$
since $\pi/2<mr_m^+<\pi$, we see that also $r_m^+>r_{m+1}^+$.
\end{proof}

\begin{cor} \label{cor:main}
$(-2,2)\cup \ri (-2,2)\subset \intt(\Sigma_\pi)$.
\end{cor}
\begin{proof} From Corollary \ref{cor:intpts} and Lemma \ref{lem:inter} it follows that $[0,\lambda_3^+)\subset \intt(\Sigma_\pi)$ and that $[\lambda_m^+,\lambda_{m+1}^+)\subset \intt(\Sigma_\pi)$, for $m\geq 3$. Thus $[0,\lambda_m^+)\subset \intt(\Sigma_\pi)$ for $m\geq 3$, and so $[0,2)\subset \intt(\Sigma_\pi)$ since $\lambda_m^+\to 2$ as $m\to\infty$. Applying Lemma \ref{lem:symeasy} we obtain the stated result.
\end{proof}

The following lemma follows immediately from Corollary \ref{cor:main}, \eqref{eq:per}, and \eqref{eq:pply}.

\begin{lem} \label{lem:pinint}
Suppose that $k\in \cK_n$, so that $k$ has length $n$ and $p_k\in \cS$ has degree $n$. Then all except at most $2n$ points in $\spec A_k^{\per}$ are interior points of $\Sigma_\pi$. Further, if $\lambda\in \spec A_k^{\per}$ then there exists a sequence $(\lambda_m)\subset \spec A_k^{\per}\cap \intt(\Sigma_\pi)$ such that $\lambda_m\to \lambda$ as $m\to\infty$.
\end{lem}

As an example of the above lemma, suppose that $k=(-1,1)\in \cK_2$. Then (see Table \ref{table1}) $p_k(\lambda)=\lambda^2$ and, from \eqref{eq:per}, $\spec A_k^{\per} = \{x\pm \ri x:-1\leq x\leq 1\}$. There are precisely four points, $\pm 1\pm \ri\in \spec A_k^{\per}\setminus \intt(\Sigma_\pi)$. These are not interior points of $\Sigma_\pi$ since they lie on the boundary of $\overline{\Delta}\supset \Sigma\supset \Sigma_\pi$.

Combining the above lemma with Theorem \ref{thm:dense}, we obtain the last result of this section.

\begin{thm} \label{thm:interior}
$\Sigma_\pi$ is the closure of its interior.
\end{thm}
\begin{proof}
Suppose $\lambda \in \Sigma_\pi$. Then, by Theorem \ref{thm:dense}, $\lambda$ is the limit of a sequence $(\lambda_n)\subset \pi^\cS_\infty$, and, by Lemma \ref{lem:pinint}, for each $n$ there exists $\mu_n\in \intt(\Sigma_\pi)$ such that $|\mu_n-\lambda_n| < n^{-1}$, so that $\mu_n\to \lambda$ as $n\to\infty$.
\end{proof}

\section{Filled Julia sets in $\Sigma_\pi$} \label{sec:julia}

It was shown in \cite{Hagger:symmetries} that, for every polynomial symmetry $p\in \cS$, the corresponding Julia set $J(p)$ satisfies $J(p)\subset \overline{U(p)}\subset \Sigma_\pi$, where $U(p)$ is defined by \eqref{eq:Up}. (The argument in \cite{Hagger:symmetries} is that $J(p) \subset \overline{U(p)}$ by \eqref{eq:Jpi}, and that $\overline{U(p)}\subset \Sigma_\pi$ by \eqref{eq:pply}.) It was conjectured in \cite{Hagger:symmetries} that also the filled Julia set $K(p)\subset \overline{U(p)}\subset \Sigma_\pi$, for every $p\in \cS$.
In this section we will first show by a counterexample that this conjecture is false; we will exhibit a $p\in \cS$ of degree 18 for which $K(p)\not\subset \overline{U(p)}$. However, we have no reason to doubt a modified conjecture, that $K(p)\subset \Sigma_\pi$, for all $p\in \cS$. And the main result of this section will be to prove that $K(p)\subset \Sigma_\pi$ for a large class of $p\in \cS$, including $p=P_m$, for $m\geq 2$.

Our first result is the claimed counterexample.

\begin{lem} \label{lem:counter} Let $k = (1,-1,1,1,1,-1,1,-1,-1,1,-1,1,1,1,-1,1,-1,1)$, so that (by Proposition \ref{prop:eqchar}) $p_k\in \cS$, this polynomial given explicitly by
$$
p_k(\lambda) = \lambda^{18} - 4\lambda^{16} + 5\lambda^{14} - 4\lambda^{12} + 7\lambda^{10} - 8\lambda^8 + 6\lambda^6 -
4\lambda^4 + \lambda^2.
$$
Then $K(p_k)\not\subset \overline{U(p_k)}$.
\end{lem}
\begin{proof} Let $p=p_k$. If we can find a $\mu\not\in \ovD$ that is an attracting fixed point of $p$, then, for all sufficiently small $\epsilon>0$, $N := \mu + \epsilon \D$ satisfies $p(N)\subset N$ and $N\cap \ovD = \emptyset$, so that $N\subset K(p)$ and $N\cap \overline{U(p)}=\emptyset$. Calculating in double-precision floating-point arithmetic in Matlab we see that $\lambda\approx 1.21544069$ appears to be a fixed point of $p$, with
$$
p^\prime(\lambda) = 18\lambda^{17} - 64\lambda^{15} + 70\lambda^{13} - 48\lambda^{11} + 70\lambda^{9} - 64\lambda^7 + 36\lambda^5 -
16\lambda^3 + 2\lambda \approx -0.69,
$$
so that this fixed point appears to be attracting. To put this on a rigorous footing we work in exact arithmetic to deduce,  by the intermediate value theorem, that $p(\lambda)=\lambda$ has a solution $\lambda^*\in(1.215,1.216)$, and that $|p^\prime(1.2155)|\leq 0.71$. Then, noting that
$p^{\prime\prime}(\lambda)= p_+(\lambda)-p_-(\lambda)$, where
$p_+(\lambda) = 306\lambda^{16} + 910\lambda^{12}  + 630\lambda^8 + 180\lambda^4 + 2$ and $p_-(\lambda) = 960\lambda^{14} + 528\lambda^{10} + 448\lambda^6 + 48\lambda^2$,
we see that
\[|p^{\prime\prime}(\lambda)|\leq \max\{|p_-(1.216)-p_+(1.215)|,|p_-(1.215)-p_+(1.216)|\} < 400\]
for $1.215\leq \lambda\leq 1.216$. But this implies that
$|p^\prime(\lambda^*)| \leq |p^\prime(1.2155)| + 0.0005 \times 400 \leq 0.91$, so that $\lambda^*$ is an attracting fixed point.
\end{proof}
Numerical results suggest that amongst the polynomials $p\in \cS$ of degree $\leq 20$, there is only one other similar counterexample of a polynomial with an attracting fixed point outside the unit disk, the other example of degree 19.

We turn now to positive results. Part of our argument will be to show, for every $p\in \cS$, that $\{z:|z|\geq 2\}\subset A_p(\infty)$, via the following lower bounds that follow immediately from Lemma \ref{lem:qkbound}, \eqref{eq:pk2} and \eqref{p_k}.

\begin{cor} \label{cor:pkbound}
If $k\in \{\pm 1\}^n$, for some $n\in \N$, then $|p_k(\lambda)|\geq 2$, for
$|\lambda|\geq 2$. If $p_k\in \cS$, then $|p_k(\lambda)|\geq 2n$, for
$|\lambda|\geq 2$.
\end{cor}

\begin{cor} \label{cor:pkjul}
Let $p=p_k$, where $k\in \{\pm 1\}^n$, for some $n\in \N$. Then $A_p(\infty) \supset \{z\in \C:|z|>2\}$. If $p\in \cS$, then $A_p(\infty)\supset \{z\in \C:|z|\geq 2\}$.
\end{cor}
\begin{proof}
Let $z\in \C$ with $|z|>2$. Then, by Corollary \ref{cor:pkbound}, for some neighbourhood $N$ of $z$, $|p_k(w)|\geq 2$ for $w\in N$. Thus, and by Montel's theorem \cite[Theorem 14.5]{Falconer03}, the family $\{p^n:n\in\N\}$ is normal at $z$. So $z\not\in J(p)$, by \eqref{eq:nn}. We have shown that $J(p)\cap \{z:|z|>2\}=\emptyset$, so that also $K(p)\cap \{z:|z|>2\}=\emptyset$ and $A_p(\infty)=\C\setminus K(p) \supset \{z:|z|> 2\}$.

If $p\in \cS$ and $|z|=2$ then, by Corollary \ref{cor:pkbound}, $|p(z)|\geq 4$ so that $p(z)\in A_p(\infty)$ and so $z\in A_p(\infty)$. Thus $A_p(\infty)\supset \{z:|z|\geq 2\}$.
\end{proof}

We remark that the bounds in Corollary \ref{cor:pkbound} appear to be sharp. In particular, if $k=(1,1,...,1)$ has length $n\geq 2$, we see from \eqref{eq:pdef2}, \eqref{p_k}, and Lemma \ref{lem:cheb} that $p_k(2)=U_n(1)-U_{n-2}(1)=2$, since $U_m(1)=m+1$ \cite{AS}. And we note that, if $p=P_m$, for some $m\in \N$, then $p(2)=P_m(2)=2U_{m-1}(1)=2m$. Finally, we recall that we have already noted that, for $p=p_k$, with $k=(1,1)$, i.e., $p(z)=z^2-2$, the Julia set is $J(p)=[-2,2]$, so that $A_p(\infty)\not\supset \{z:|z|\geq 2\}$ for this $p$.

The polynomial $p(z)=z^2-2$ is an example where $J(p)=K(p)$ so $F_B(p)=K(p)\setminus J(p) =\emptyset$. The next lemma tells us that this does not happen, that $K(p)$ is strictly larger than $J(p)$, if $p\in\cS$.

\begin{lem}\label{lem:larger}
$F_B(p)\cap U(p)$ is non-empty for $p\in \cS$.
\end{lem}
\begin{proof}
If $p\in \cS$ is even then, by Lemma \ref{lem:det} and Corollary \ref{cor:small}, $p(0)=p^\prime(0)=0$, so that $0$ is an attracting fixed point. Clearly $A_p(0)$ (which is non-empty) is a subset of $U(p)\cap F_B(p)$. Similarly, by Lemma \ref{lem:det} and Corollary \ref{cor:small}, if $p$ is odd then $p(0)=0$ and $p^\prime(0)=\pm 1$, so that $0\in J(p)$ is a rationally neutral fixed point and has a (non-empty) attracting region contained in $F_B(p)$ \cite[Section II.5]{CarlesonGamelin92}, this region clearly also in $U(p)$.
\end{proof}

The above lemma and \eqref{eq:pply} imply that $F_B(p)\cap \Sigma_\pi\supset F_B(p)\cap U(p)$ is non-empty for all $p\in \cS$, in particular that $A_p(0)\subset F_B(p)\cap U(p) \subset \Sigma_\pi$ if $p$ is even. The main result of this section is the following criterion for the {\em whole} of $F_B(p)$ to be contained in $\Sigma_\pi$.

\begin{thm} \label{prop:filled3}
Suppose that $p\in \cS$, and that the critical points of $p$ in $K(p)$ have orbits that lie eventually in $1.1\ovD\cup (-2,2)\cup \ri(-2,2)$.  Then $K(p)\subset \Sigma_\pi$.
\end{thm}
\begin{proof} Choose $a$ and $b$ with $-2<a<b<2$ such that $K(p)\cap \R \subset [a,b]$ and $K(p)\cap \ri\R \subset \ri[a,b]$, this possible by Corollary \ref{cor:pkjul} which says that the closed set $K(p)\subset \{z:|z|<2\}$. Set $T = [a,b]\cup\ri[a,b]\cup 1.1\ovD$, and choose a simply-connected open set $S$ such that $T\subset S\subset \Sigma_\pi$, this possible by Corollary \ref{cor:main} and Proposition \ref{prop:betterbounds}. By hypothesis, the orbits of the critical points in $K(p)$ lie eventually in $T$. Thus the lemma follows from Proposition \ref{prop:cd2a} and \eqref{eq:pply}.
\end{proof}

\begin{figure}
\begin{center}
\includegraphics[scale = 0.7, trim = 0 1cm 0 0]{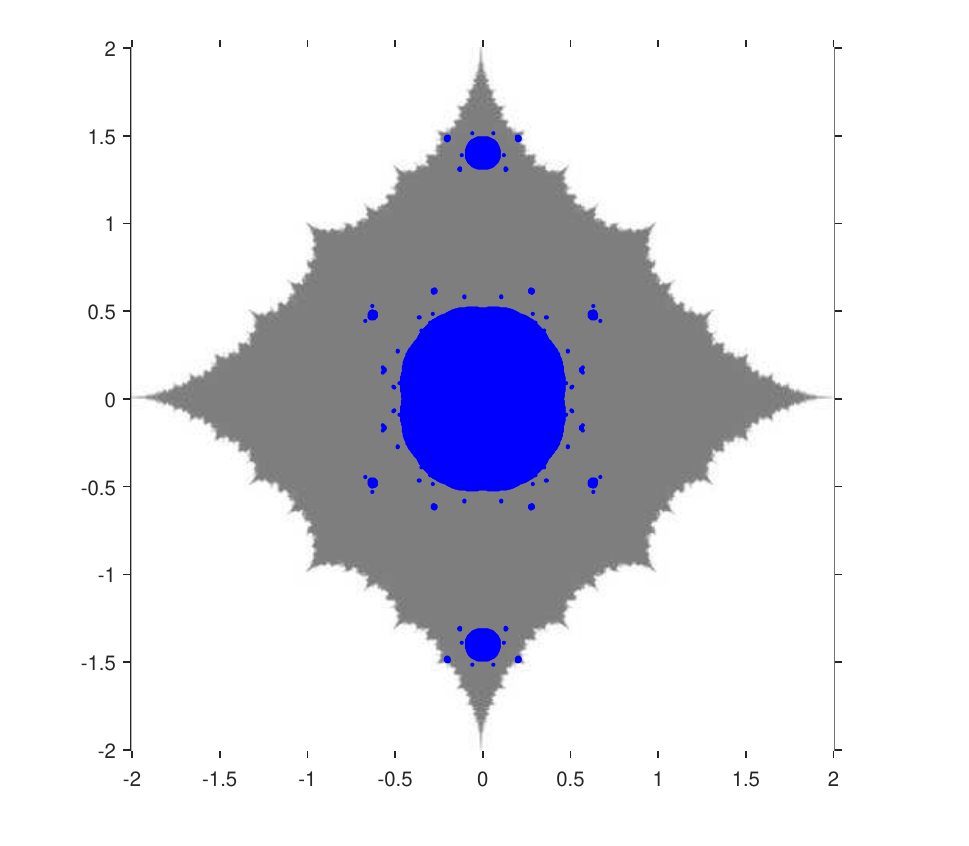}

\caption{Plot of $\pi_{30}\subset \Sigma_\pi$ and (in blue) the filled Julia set $K(p)$ in the case $p(\lambda) = P_4^*(\lambda) = \lambda^4 + 2\lambda^2$. By Corollary \ref{cor:Kp}, $K(p)\subset \Sigma_\pi$.} \label{Figure3}
\end{center}
\end{figure}

As an example of application of this theorem, consider $p\in \cS$ given by (see Table \ref{table1})
$p(\lambda) = P_4^*(\lambda)=\lambda^4+ 2\lambda^2$. This $p$ has critical points $0$ and $\pm \ri$. Since $p^2(\pm \ri) = 3$ it follows from Corollary \ref{cor:pkjul} that $\pm \ri\in A_p(\infty)$, while $0$ is a fixed point. Theorem \ref{prop:filled3} tells us that $K(p)$, visualised in Figure \ref{Figure3}, is contained in $\Sigma_\pi$.
We note that, since all the critical points of $p$ except the fixed point $0$ are in $A_p(\infty)$, $K(p)$ is not connected \cite[Theorem III 4.1]{CarlesonGamelin92} and, by Theorem \ref{thm:sullivan} and Proposition \ref{prop:crit}, $F_B(p)=A_p(0)$, which implies that $K(p)\subset \overline{U(p)}$. Further, recalling the discussion in Section \ref{sec:pre}, $J(p) =\partial K(p)=\partial A_p(0)=\partial A_p(\infty)$, and, since $K(p)$ has more than one component, $F_B(p)$ has infinitely many components \cite[Theorem IV 1.2]{CarlesonGamelin92}.

 The above example is a particular instance of a more general result. It is straightforward to see that if $p$ is a polynomial with zeros only on the real line, then all the critical points are also on the real line.  Since, by Lemma \ref{lem:cheb}, $P_m(\lambda) = \lambda U_{m-1}(\lambda/2)$, and all the zeros of the polynomial $U_{m-1}$ are real, it follows that all the zeros of $P_m$ are real, so all its critical points are also real, and so the orbits of all the critical points are real. Further, by Corollary \ref{cor:pkjul}, the orbits of the critical points in $K(p)$ stay in $(-2,2)$. Likewise, as (see \eqref{eq:PM*}) $P_m^*(\lambda)=\ri^{-m}P_m(\ri\lambda)$, all the critical points of $P_m^*$ lie on $\ri\R$, and so the orbits of these critical points are real if $m$ is even, pure imaginary if $m$ is odd. Further, by Corollary \ref{cor:pkjul}, the orbits of the critical points in $K(p)$ stay in $(-2,2)\cup \ri(-2,2)$. Applying Theorem \ref{prop:filled3} we obtain:

\begin{cor} \label{cor:Kp} $K(P_m)\subset \Sigma_\pi$ and $K(P_m^*)\subset \Sigma_\pi$, for $m\geq 2$.
\end{cor}

Numerical experiments carried out for the polynomials $p\in \cS$ of degree $\leq 7$ (see Table \ref{table1} and \cite[Table 1]{Hagger:symmetries}) appear to confirm that these polynomials satisfy the conditions of Theorem \ref{prop:filled3}, i.e., it appears for each polynomial $p$ that the orbit of every critical point either diverges to infinity or is eventually in $1.1\ovD\cup (-2,2)\cup \ri(-2,2)$. The same appears true for the polynomial $p\in \cS$ of degree 18 in Lemma \ref{lem:counter} for which $K(p)\not\subset \overline{U(p)}$. Thus it appears, from numerical evidence and Theorem \ref{prop:filled3}, that $K(p)\subset \Sigma_\pi$ for these examples. These numerical experiments and Corollary \ref{cor:Kp} motivate a conjecture that $K(p)\subset \Sigma_\pi$ for all $p\in \cS$.

\section{Open Problems} \label{sec:open}
We finish this paper with a note of open problems regarding the spectrum of the Feinberg-Zee random hopping matrix, particularly problems that the above discussions have highlighted. We recall first that \cite{CCL} made several conjectures regarding $\Sigma$. It was proved in \cite{Hagger:dense} that $\overline{\sigma_\infty}=\Sigma_\pi$, but the following conjectures remain open:
\begin{enumerate}
  \item $\Sigma_\pi=\Sigma$;
  \item $\Sigma$ is the closure of its interior;
  \item $\Sigma$ is simply connected;
  \item $\Sigma$ has a fractal boundary.
\end{enumerate}
Of these conjectures, perhaps the first has the larger implications. Certainly, if $\Sigma=\Sigma_\pi$, then we have noted below \eqref{eq:convPin} that we have constructed already convergent sequences of upper ($\Sigma_n^*$) and lower ($\Pi_n$) bounds for $\Sigma$ that can both be computed by calculating eigenvalues of $n\times n$ matrices. Further, if $\Sigma=\Sigma_\pi$, then the second of the above conjectures follows from Theorem \ref{thm:interior}.

The last three conjectures in the above list were prompted in large part by plots of $\pi_{n}$ in \cite{CCL}, the plot of $\pi_{30}$ reproduced in Figures \ref{Figure2} and \ref{Figure3}. It is plausible that these plots, in view of \eqref{eq:convPin}, approximate $\Sigma_\pi$. We see no clear route to establishing the third conjecture above. Regarding the fourth, we note that the existence of the set $\cS$ of polynomial symmetries satisfying \eqref{eq:spsym} suggests a self-similar structure to $\pi_\infty$ and to $\Sigma_\pi$ and $\Sigma$  and their boundaries. Further, \cite{Hagger:symmetries} has shown that $\Sigma_\pi$ contains the Julia sets of all polynomials in $\cS$, and Proposition \ref{prop:filled3} and Corollary \ref{cor:Kp} show that $\Sigma_\pi$ contains the filled Julia sets, many of which have fractal boundaries, of the polynomials in an infinite subset of $\cS$.

Regarding these polynomial symmetries we make two further conjectures:
\begin{enumerate}
  \item[5.] $K(p)\subset \Sigma_\pi$ for all $p\in \cS$;
  \item[6.] $p^{-1}(\Sigma)\subset \Sigma$ for all $p\in \cS$.
\end{enumerate}
This last conjecture follows if $\Sigma=\Sigma_\pi$, by Theorem \ref{thm:cS} from \cite{Hagger:symmetries}. Further (see the discussion below \eqref{eq:pply}), it was shown in \cite{CWDavies2011} that $p^{-1}(\Sigma)\subset \Sigma$ for the only polynomial of degree 2 in $\cS$, $p(\lambda)=\lambda^2$.

The major subject of study and tool for argument in this paper has been Hagger's set of polynomial symmetries $\cS$. We finish with one final open question raised immediately before Section \ref{subsec:a}.
\begin{enumerate}
  \item[7.]  Does $\cS$ capture all the polynomial symmetries of $\Sigma$? Precisely, are there polynomial symmetries, satisfying \eqref{eq:spsym}, that are not either in $\cS$ or compositions of elements of $\cS$?
\end{enumerate}






\subsection*{Acknowledgments}
We thank our friend and scientific collaborator Marko Lindner for introducing us to the study of, and for many discussions about, this beautiful matrix class.
\end{document}